\documentclass[a4paper,11pt]{article}
\usepackage{a4wide}
\usepackage[english]{babel}
\usepackage{amssymb}
\usepackage{amsmath}
\usepackage{amsfonts}
\usepackage{amsthm}
\usepackage[pdftex]{graphicx}
\usepackage{hyperref}
\hypersetup{colorlinks,citecolor=red,filecolor=purple,linkcolor=blue,urlcolor=black}
\usepackage[T1]{fontenc}
\usepackage{enumerate}
\usepackage{relsize}
\usepackage{mathtools}
\usepackage{bbm}

\newcommand{\R}{\mathbb R}
\newcommand{\E}{\mathbb E}

\newcommand{\p}{\mathbb P}

\newcommand{\Var}{\mathrm{Var}}
\newcommand{\me}{\medskip \noindent}
\newcommand{\bi}{\bigskip \noindent}
\newcommand{\un}{\mbox{\large$\mathbbm{1}$}}
\newcommand{\e}{\mathsmaller{E}}
\newcommand{\D}{\mathsmaller{D}}
\def\bl{ \textcolor{blue}}

\def\be{\begin{eqnarray}}
\def\ee{\end{eqnarray}}
\def\ben{\begin{eqnarray*}}
\def\een{\end{eqnarray*}}

\newtheorem{prop}{Proposition}[section]

\newtheorem*{assumptionA}{Assumption A}
\newtheorem{lem}[prop]{Lemma}
\newtheorem{thm}[prop]{Theorem}
\newtheorem{rem}[prop]{Remark}
\newtheorem{cor}[prop]{Corollary}
\newtheorem*{assumptionA1}{Assumption A1}
\newtheorem*{assumptionA2}{Assumption A2}
\newtheorem*{assumptionA2'}{Assumption A2'}

\newtheorem*{assumptionA3}{Assumption A3}
\newtheorem*{assumptionA1'}{Assumption A1'}
\newcounter{example}
\title{Scaling limits of  population and evolution  processes  in random environment}

\author{Vincent Bansaye\thanks{CMAP, Ecole Polytechnique, CNRS, route de Saclay, 91128 Palaiseau Cedex-France; E-mail: \href{mailto:vincent.bansaye@polytechnique.edu}{\texttt{vincent.bansaye@polytechnique.edu} }},
    Maria-Emilia Caballero\thanks{UNAM; E-mail: \href{mailto:mariaemica@gmail.com}{\texttt{mariaemica@gmail.com}}},  Sylvie M\'el\'eard\thanks{CMAP, Ecole Polytechnique, CNRS, route de
   Saclay, 91128 Palaiseau Cedex-France; E-mail: \href{mailto:sylvie.meleard@polytechnique.edu}{\texttt{sylvie.meleard@polytechnique.edu}}}}

\begin{document}

\maketitle

\begin{abstract} Our motivation comes from the large population approximation of individual based models in population dynamics and population genetics. We propose a general method to investigate scaling limits of finite dimensional   Markov chains   to diffusions with jumps. 
The results of tightness, identification and convergence in law  are based on the convergence  of  suitable  characteristics of the  chain transition. They strongly exploit the structure of the population processes recursively defined  as sums of independent random variables.  We develop  two main applications.
First, we extend the classical Wright-Fisher diffusion approximation to independent and identically distributed random environments.
  Second, we  obtain the convergence in law of generalized Galton-Watson processes with interaction  in random environment to  the solution of  stochastic differential equations with jumps. \end{abstract}


\medskip

\setcounter{tocdepth}{2} 

\section{Introduction}
This work is a contribution to the study of  scaling limits of discrete population models. The parameter  $N\in \mathbb N$ will scale  the population sizes. The population processes  $(Z_n^N : n\in \mathbb N)$  are $\mathbb{N}^d$-valued Markov chains inductively    defined by  
$$Z^{N}_{n+1}=\sum_{j=1}^{F_N(Z^N_n)} L^N_{j,n}(Z^N_n,E^N_n),$$
where $F_N$ is a function giving the number of individual events. For each $z,e,N$,  $(L^N_{i,n}(z,e) : i,n \geq 1)$ is a family of independent identically distributed random variables
and $E^N_n$ is a $\R^d$-random variable describing the environment at generation $n$. This class of processes includes well known processes in population dynamics and population genetics. In particular,  
Galton-Watson processes correspond to $F_N(z)=z$ and $L^N_{i,n}(z,e)$ does not depend on  $(z,e)$, while Wright-Fisher processes  are obtained with    $F_N(z)=N$ and $L^N_{i,n}(z,e)$ Bernoulli random variables with  parameter $z/N$. This class of population models  can also take into account the effect of random environment and  include many additional ecological forces such as competition,  cooperation and  sexual reproduction. 

\me In this paper we are interested in the convergence of the sequence of  processes $(Z_{[v_Nt]}^N/N : t\geq 0)$,  as $N$ tends to infinity, $v_N$ being   a  time scale tending to infinity with $N$. We  provide a unified framework adapted to population models and  characterize this convergence by asymptotic properties relying on $v_{N}, F_{N}$ and $L^N$.  Many works have been devoted to the approximation of  Markov processes. They are essentially based on tightness argument and identification of the martingale problem, see for example \cite{EK, Jacod}. Unfortunately, this general method does not satisfactorily apply to our framework since the required  assumptions are difficult to check. Applying for instance  this  method to  the classical Galton-Watson framework leads to moment assumptions.  However, it is well known from the works of Lamperti \cite{Lamp, Lamp2} and Grimvall \cite{Grimvall} that  the finite dimensional convergence  of 
the renormalized processes $(Z_{[v_Nt]}^N/N : t\geq 0)$ with a  time scale $v_N\rightarrow \infty$ is equivalent to the convergence of a characteristic triplet associated with $(v_{N},L^N)$  when $N$ tends to infinity.  In this case, the sequence of  processes  $(Z_{[v_Nt]}^N/N : t\geq 0)$ converges as $N\rightarrow \infty$ to a Continuous State Branching Process (CSBP) defined as the unique strong solution of a Stochastic Differential Equation (SDE). The parameters of this SDE are given by the limiting characteristic triplet of $(v_{N},L^N)$. 
Note that the  proof is based on the branching property, using either   the Laplace exponent \cite{Grimvall}, or the relation with  the convergence of the associated random walk to a spectrally positive L\'evy process via  a Lamperti time change (cf. \cite{Lamp2} \cite{CLU}). Lamperti has also introduced a powerful transform  in the stable framework, see e.g. \cite{Lamp2}  and   \cite{Li} and  \cite{BK}.
Other time changes have been successfully used to obtain scaling limits of discrete processes, in particular  for some diffusion approximations, see for  instance \cite{Kurtz78} for branching
processes in random environment, \cite{CPU} for branching processes with immigration  and  \cite{Rosen} for controlled branching processes , amongst others.
Such time change  techniques seem essentially restricted to branching processes or stable processes or diffusion approximations. We are here interested in the convergence in law of  discrete Markov processes  $(Z^N)_{N}$ which do not  enjoy the branching property and may jump in the limit. 
The limiting processes may even be explosive and are not necessarily stable. \\

It is well known that the law of  the process $(Z^N_{[v_N.]}/N)$  is determined  by its initial law and  the family of functions
$$x\rightarrow {\cal G}^N_x(H)=v_N\E\left( H(Z^N_1/N - x )| Z^N_0= Nx\right)$$
   for $H$ continuous and bounded  on $\mathbb{R}^d$. 
Moreover, the asymptotic behavior of $ {\cal G}^N_x(H)$ as $N\rightarrow \infty$ for a  large enough class of functions $H$  captures the convergence of the  processes. 
 In such discrete case, Jacod  and Shiryaev  in \cite[II.3, IX]{Jacod} prove that the   
  tightness  and the identification are deduced from the convergence as $N\rightarrow \infty$ of the characteristics of the semimartingales 
$$\sum_{i\leq [v_Nt]}  {\cal G}^N_{Z^N_i/N} (H)$$
defined for certain  functions $H$: a truncation (vector) function and its squares and a determining class of smooth functions vanishing in a neighborhood of $0$. Let us  mention  for example \cite{Kosenkova} where the convergence of Markov chains to a L\'evy driven SDE is shown. 
This strategy is unfortunately difficult to apply in our framework, even for Galton-Watson processes. We prove in this paper that the functions $H$ can be chosen  differently, belonging to some (rich enough) functional space ${\cal H}$,  dense in the set of regular functions vanishing at zero for a norm equivalent to $$\| H\| = \sup_{u\in\mathbb{R}^d}\Big\|  {H(u)\over 1\wedge |u|^2} \Big\|_{\infty}.$$
The  choice of the space ${\cal H}$ will depend on  the assumptions on the model. In our applications, we   exploit  the independence property of the variables $(L^N_{i,.}(.,.) : i \geq 1)$. We  plug at the level of the conditional increments the characterization of the law by the Laplace exponent, which is well adapted to the sum of non-negative independent random variables.\\

Our main motivations were   the famous frameworks of population genetics and dynamics.
We show   the efficiency of our method on extensions of  Wright-Fisher and Galton-Watson chains. We first study a  Wright-Fisher model with selection in a random environment  impacting the selective advantage. 
  The environments  are assumed to be independent and identically distributed and  the associated random walk converges to a L\'evy process. We obtain the convergence of the joint law  of the processes and   random walks, by using the functional space 
 $${\cal H}=\{ (u,w) \in [-1,1]\times (-1,\infty) \rightarrow  1 - e^{- k u -\ell w}  \, ; \,  k,   \ell \geq 0\}.$$
 We thus derive a  diffusion with jumps in random environment, which generalizes the Wright-Fisher diffusion with selection and takes into account  small random fluctuations and punctual dramatic advantages in the selective effects.\\
   The second application focuses on generalized Galton-Watson processes with  reproduction law both density dependent and environment dependent.  
 We obtain a result of convergence in law to the so called \emph{CSBP with interaction in L\'evy environment} (introduced in \cite{PP, ZL}). For such processes, characteristics are unbounded  and the result is deduced  from the convergence of the compactified processes
$$\exp(-Z_k^N/N).$$
To deal with the joint laws of the latter and  the environment random walk, we use the  space of functions from  $[-1,1]\times (-1,\infty)$ to $\R$ defined by
$$\mathcal H=\{ (v,w) \rightarrow v^k\exp(-\ell w): k\geq 1, \ell \geq 0 \} \cup \{ (v,w) \rightarrow 1-\exp(-\ell w) :  \ell \geq 1\}.$$
Our results extend the criterion for the convergence of a sequence of Galton-Watson processes and the results we know in  random environment  \cite{Kurtz78, BS}
or with interactions  \cite{PD, Pardoux}. \\
They are further applied to Galton-Watson processes  with cooperation and to branching processes  with logistic growth in random environment. \\

 {\bf{Organization of the paper}}\\ 
 In Section \ref{gene}, we give general results for  the tightness, the identification and the convergence in law of a scaled Markov process to a diffusion with jumps in $\R^d$. The functional space ${\cal H}$ is introduced in Section \ref{funct}.
Tightness and identification results are stated in Section \ref{statgen}
by assuming 
 the uniform convergence and boundedness of characteristics  ${\cal G}^N_.(H)$ for any $H\in {\cal H}$. Convergence requires an additional uniqueness assumption,   obtained from pathwise uniqueness in the applications, using standard techniques for non-negative SDE  \cite{IK, FL}. Proofs of these general statements are given in Section \ref{proofgene}. In Section \ref{WF}, we apply our method  to a Wright-Fisher model with selection in a random environment. We  obtain in a suitable scaling limit a Wright-Fisher diffusion in random environment for which we prove uniqueness of solution.   In Section \ref{CSBPLEI} (Sections 4.1, 4.2, 4.3), we apply our method  to  Galton-Watson processes with  reproduction law both density dependent and environment dependent. Section 4.4 is devoted to explosive  CSBP with interaction in random environment. In particular, we consider Galton-Watson processes with cooperative effects.  Section 4.5 is dedicated to the conservative case and an application to Galton-Watson processes with logistic competition and small environmental fluctuations is studied.  Finally, 
we expect  the method to be applied in various contexts, in particular    for structured  populations models
with sexual reproduction, competition or cooperation, see Section \ref{Perspect}. \\

\noindent {\bf Notation}\\
For $x\in \R^{d}$, we denote  by $|x|$ the euclidian norm of $x$. If $A\subset \mathbb{R}^d$,  $\overline{A}$ is the closure of $A$ in $\mathbb{R}^d$.
\\
The functional norms are denoted by $\|.\|$. In particular the sup norm of a bounded function $f$ on a set ${\cal U}$ is denoted by $\|f\|_{{\cal U}, \infty}$. The sets $C_{b}({\cal U}, \mathbb{R})$ and  $C_{c}({\cal U}, \mathbb{R})$ denote the spaces of continuous real functions defined on ${\cal U}$ respectively  bounded and with compact support.  \\
As usual, we write 
 $h(u)=o(g(u))$ (resp. $h(u)\sim g(u)$) when $h(u)/g(u)$ tends to $0$  (resp. to $1$) as $u$ tends  to $0$.    $Id$ denotes the identity function.\\ 
For any $\mathcal U$    subset of $\mathbb{R}^d$  containing  a neighborhood of $0$, we define $\mathcal U^*$ as $\mathcal U\setminus \{0\}$.

\section{A criterion for tightness and convergence in law}
\label{gene}

Let  ${\cal X}$   be  a Borel subset of $\mathbb{R}^d$ and $\mathcal U$   be a closed subset of $\mathbb{R}^d$  containing  a neighborhood of $0$. 

\me 
Let us introduce a scaling parameter $N\geq 1$ .  For any $N$, we consider 
a discrete time ${\cal X}$-valued   Markov chain $(X^N_{k} :  k\in \mathbb{N})$ satisfying
for any $k\geq 0$,
$$ {\cal L}(X^N_{k+1} \,| \, X^N_{k}= x) = {\cal L}(F^N_{x}),$$
 where for any $N\in \mathbb N$,  $(F^N_{x}, x \in {\cal X})$ denotes a measurable family of ${\cal X}$-valued random variables    such that 
 for any $x\in  {\cal X}$, the random variable $F^N_{x} - x$ takes values in   ${\cal U}$.

\me The natural filtration of the process $X^N$  is denoted by
$({\cal F}^N_{k})_k.$ Note that the increments $X^N_{k+1} - X^N_{k}$ take values in $\mathcal U$. 

\me Our aim is the characterization of the convergence in law of the sequence of processes $\,(X^N_{[v_{N}.]}, N\in \mathbb{N})$, where $(v_{N})_{N}$ is a given sequence of positive real numbers going to infinity when $N$ tends to infinity. It is based on the criteria for   tightness and identification of semimartingales by use of characteristics    given in   \cite[IX]{Jacod}, which 
consists in studying the asymptotic behavior of
\be\label{GNx}
{\cal G}_{x}^N(H) = v_{N}\, \E\big(H(F^N_{x}-x)\big) = v_{N}\, \E\big(H(X^N_{k+1} - X^N_{k})\,| \, X^N_{k}= x\big),\ee
for real valued bounded measurable functions $H$ defined on ${\cal U}$.

\bi  {\bf Hypothesis (H0)}
 \emph{ We first assume that the family of random variables $(F^N_{x})_{N,x}$  satisfies
\ben
\lim_{b\to \infty} \sup_{x\in {\cal X}, N\in \mathbb{N}^*} {\cal G}_{x}^N\big(\un_{B(0,b)^c} \big)= 0. 
\een}

This hypothesis avoids to get infinite jumps in the limit. We will see in the examples  that this condition affects both the population and the environment dynamics.

\me Under  {\bf (H0)}, we will prove that the study of \eqref{GNx} can be reduced to a rich enough and tractable  subclass ${\cal H}$ of functions $H$. The choice of ${\cal H}$
 depends on the particular models and is illustrated in the examples.

\subsection{Specific and truncation functions}
\label{funct}
We  consider a closed subset $\mathcal U$ of $ \R^d$ containing a neighborhood of $0$ and   introduce the functional space
$$C_{b,0}^2 = C^2_{b,0}({\cal U}, \mathbb{R})=  \left\{ H\in C_{b}( {\cal U}, \mathbb{R}) \, :  \, H(u) = \sum_{i=1}^d \alpha_{i}u_{i} + \sum_{i,j=1}^d \beta_{i,j} u_{i} u_{j}+ o(|u|^2), \, \alpha_{i}, \beta_{i,j} \in \R \right\}.$$

The  functions of   $C^2_{b,0}$ can  be decomposed in a similar way with respect to any smooth function which behaves like the  identity at $0$, 
as stated in the next lemma. The proof uses the uniqueness of the second order Taylor expansion in a neighborhood of $0$.
 \begin{lem}
\label{lem-decomp}
Let  $f=(f^1, \ldots, f^d) \in (C^2_{b,0})^d$ such that $f^i(u)=u_i(1+o(\vert u\vert))$ for $i=1, \ldots, d$. For any 
$H\in C^2_{b,0}$,  there exists a unique decomposition of the form
$$ H = \sum_{i=1}^d \alpha^f_{i}(H) f^{i} + \sum_{i,j=1}^d \beta^f_{i,j}(H)  f^{i} f^{j}+\overline{H}^f,$$ where
 $\overline{H}^f= o(|f|^2)$ is a continuous and bounded function and $\alpha_{i}^f(H)$, $\beta_{i,j}^f(H)$, $i,j=1 \cdots d$ are real coefficients and $\beta^f$ is a symmetric matrix.
 \end{lem}
  \me We introduce  
 \begin{itemize}  
 \item   the  {\bf{specific function}}  $h$ which  satisfies
  \be \label{truncation} h = (h^1, \cdots,  h^d) \in  (C^2_{b,0})^d\ ;\ h^i(u)=  u_{i}(1+o(u))\ ;\ h^i(u)\neq 0 \ {\rm for }\  u  \neq 0 \quad (i=1, \ldots,d).\ee
\item  the  {\bf{truncation function}}  $h_0$, as defined in \cite{Jacod}  :
   \be
\label{troncationjacod}
h_{0}=(h_{0}^1, \cdots, h_{0}^d)\in C_b({\cal U}, \R^d), \quad  h_{0}(u)=u \  \text{ in a  neighborhood of }  \ 0.
\ee
 Obviously,  $h_0^i h_0^j\in C^2_{b,0}$ for any $ i,j=1,\ldots,  d$.  
\end{itemize}
Note that in general a specific function is not a truncation function since it may not coincide with the identity function in a neighborhood of $0$. Its choice will be driven by the processes we are  considering. We  will give different choices of functions $h$ in the next sections, for instance
$h(x)=1-\exp(-x)$ on $[-1,\infty)$ when $d=1$. These specific functions will play a crucial   role in the whole paper.

\subsection{General statements}

 \bi We introduce a functional  space $\mathcal H$
 containing the coordinates of  the specific function $h$ and their square products and  which "generates" the continuous functions with compact support in ${\cal U}$ in the sense described below. 
 The space $\mathcal H$ will be a convergence determining class.

 
 \me {\bf Hypotheses (H1)}
 There exists a functional space $\mathcal H$ such that 
 \begin{enumerate}
 \item   \emph{${\cal  H}$ is a subset of  $C^2_{b,0}$ and $\, h^i, h^ih^j\, \in Vect( {\cal H})$  for $i,j = 1, \ldots, d$.}
\item  \emph{For any  $ g\in C_{c}({\cal U}, \mathbb{R})$ with $g(0) = 0$, there exists a sequence $(g_{n})_{n}  \in C^2_{b,0}$ such that $\,\lim_{n \to \infty} \|g-g_{n}\|_{\infty, {\cal U} }= 0$
and $|h|^2\, g_n \in Vect({\cal H})$.} 

 \item    \emph{There exists a family of real numbers $\,({\cal G}_{x} H ;  x\in {\cal X}, H \in {\cal  H})$ such that for any $H \in {\cal  H}$}, 
 \ben
(i) && \qquad \lim_{N\to \infty}\sup_{x\in {\cal X}} \left| {\cal G}_{x}^N(H) - {\cal G}_{x}(H) \right| = 0. \qquad \qquad\\
(ii) &&  \qquad 
\sup_{x\in {\cal X}} \left|  {\cal G}_{x}(H) \right|  < + \infty.
\een
\end{enumerate}

\me \begin{rem} In the examples of the next sections,   {\bf (H1.2)} is proved with the use of the locally compact version of the Stone-Weierstrass Theorem.
We refer to the  Appendix for a precise statement. \\
Contrary to the "convergence determining class" of \cite{Jacod},  the functions of ${\cal H}$  will not be vanishing (or $o(u^2)$) in a neighborhood of $0$.\\
Hypothesis ${\bf (H1.3)}$ implies that the map $ x\in  {\cal X} \rightarrow{\cal G}_{x}(H) $ is measurable and  bounded for any $\,H\in {\cal H}$.
\end{rem}

\label{statgen}


We first obtain a tightness result based on the space ${\cal H}$ of test functions.

 \begin{thm}
\label{thm-tightness} Assume that the sequence $(X_0^N)_N$  is tight in $\overline{{\cal X}}$
and that {\bf (H0)} and ${\bf (H1)}$  hold. Then the  sequence of processes $\,(X^N_{[v_{N}.]}, N\in \mathbb{N})$ is tight in $\mathbb{D}([0,\infty), \overline{{\cal X}})$.
\end{thm}

\me 
 The next hypothesis {\bf (H2)} in addition to {\bf (H1)}  is sufficient to get the identification
  of the limiting values by their  semimartingale characteristics, and  then their representation  as solutions of a stochastic differential equation. 

\bi {\bf Hypotheses (H2)}   
 \begin{enumerate}
 \item    \emph{For any  $\,H\in {\cal H}$, the map $  x\in  {\cal X} \rightarrow{\cal G}_{x}(H) $ is continuous and extendable by continuity to $ \overline {\cal X}.$} 
\item  \emph{For
any $\,x\in \overline{{\cal X}}$ and any $\,H\in {\cal H}$},
\be
\label{h2}
 {\cal G}_{x}(H) =  \sum_{i=1}^d \alpha_{i}^{h_0}(H) b_{i}(x) + \sum_{i,j=1}^d  \beta_{i,j}^{h_0}(H)c_{i,j}(x)   + \int_{V} \overline{H}^{h_0}(K(x,v)) \mu(dv), 
 \ee
 \end{enumerate}
 \emph{where
 \begin{itemize}
\item[i)] 
  $\alpha_{i}^{h_0}$, $\beta_{i,j}^{h_0}$ and $\overline{H}^{h_0}$ have been defined in Lemma \ref{lem-decomp}, 
  \item[ii)]  $b_{i}$ and $\sigma_{i,j}$ are measurable functions defined  on $\overline{{\cal X}}$, 
  \item[iii)]  $V$ is a Polish space, $\mu$ is a $\sigma$-finite positive measure  on $V$, $K$ is a function from $\overline{{\cal X}}\times V$ with values in ${\cal U}$,  $ \int_{V}   1\wedge\vert K(.,v)\vert^2  \mu(dv)<+\infty$ and 
$$c_{i,j}(x)=\sum_{k=1}^d\sigma_{i,k}(x)\sigma_{j,k}(x) +\int_{V} (h^i_{0} h^j_{0})(K(x,v)) \mu(dv).$$
\end{itemize}}

 \bi The elements $(b,\sigma,V,\mu,K)$ will  be specified in the applications. 

\bi
\begin{thm}
 \label{identification}  
  If the sequence $(X_0^N)_N$  is tight in $\overline{{\cal X}}$ 
and  {\bf (H0)}, ${\bf (H1)}$,  ${\bf (H2)}$ hold  then any  limiting value of $\,(X^N_{[v_{N}.]}, N\in \mathbb{N})$  is a   semimartingale    solution of the  stochastic differential system
 \be
 \label{eds}
 X_{t}&=&X_{0} + \int_{0}^t b(X_{s}) ds + \int_{0}^t \sigma(X_{s}) dB_{s} + \int_{0}^t\int_{V} h_{0}(K(X_{s-},v)) \tilde N(ds, dv)\nonumber\\
 && \hskip 3cm + \int_{0}^t\int_{V} (Id - h_{0})(K(X_{s-},v)) N(ds, dv),
 \ee
where  
 $X_{0}\in \overline{{\cal X}}$ and $B$ is a $d$-dimensional Brownian motion and
    $N$ is a Poisson Point  measure on $\R_{+}\times V$ with intensity $ds \mu(dv)$. Moreover  $X_{0}, B$ , $N$ are independent and $\tilde N$ is the   compensated martingale measure of $N$.
    \end{thm}

 \bi
 To obtain the convergence in law of the sequence of processes $\,(X^N_{[v_{N}.]}, N\in \mathbb{N})$ in $\mathbb{D}([0,\infty), \overline{{\cal X}})$, we need 
 
 \bi
 \me {\bf Hypothesis (H3)}
\emph{The law of  the  initial condition $X_0\in \overline{{\cal X}}$ being given, 
the uniqueness in law of the solution of  (\ref{eds}) holds  in $\mathbb{D}([0,\infty), \overline{{\cal X}})$.}
 

  \me We are now in position to state the convergence result.

 \begin{thm}
 \label{main}  Assume that the sequence $(X_0^N)_N$ converges in law in $\overline{{\cal X}}$ to $X_0$
 and that   {\bf (H0)}, {\bf (H1)},  {\bf (H2)} and {\bf (H3)} hold.   Then the sequence of processes $(X^N_{[v_{N}.]})_N$ converges in law  in $\mathbb{D}([0,\infty), \overline{{\cal X}})$  to the solution of \eqref{eds}.
 \end{thm}

 \subsection{Proofs}
 \label{proofgene}
 From now on, we assume that hypotheses  {\bf (H0)} and  {\bf (H1)}  hold. We recall that $\mathcal U^*=\mathcal U\setminus\{0\}$.  
 
 \me In the proofs, we  use 
 the space  ${\cal R}_{b}$ of continuous and bounded functions which are small enough close to $0$ :
$$ {\cal R}_{b} = \{ H\in C_{b}({\cal U}, \mathbb{R}), \, H(u) = o(|u|^2)\}.$$
Using  Lemma \ref{lem-decomp} and {\bf (H1.1)}, we have
 \be
\label{decomp}
C^2_{b,0} =   Vect({\cal  H}) +  {\cal R}_{b}.\ee
We work with the norm
$$\|H\|_{h} = \sup_{u\in {\cal U}^*}  \frac{|H(u)|}{|h(u)|^2}, $$
defined for $H\in {\cal R}_b$ such that $\sup_{u\in {\cal U}^*}  |H(u)|/|h(u)|^2<+\infty$. 
In that case, the  positivity  and linearity of ${\cal G}_{x}^N$   for all $  x\in {\cal X}$ and $N\geq 1$ imply that  \be
\label{majj}
\vert {\cal G}_{x}^N(H) \vert \leq {\cal G}_{x}^N(|h|^2) \,\|H\|_{h} \leq \alpha(|h|^2)\,\|H\|_{h},
\ee
where $\alpha(|h|^2) = \sup_{N,x\in \mathcal X}| {\cal G}_{x}^N (|h|^2)|<\infty$ by ${\bf (H1.3)}$ since $\vert h\vert^2 \in Vect({\cal H})$ by ${\bf (H1.1)}$.\\
 
 \subsubsection{Proof of Theorem  \ref{thm-tightness}  }

We  first extend the assumptions {\bf (H1.3)} to $C^2_{b,0}$ in order to prove the tightness.
 We note that  {\bf (H1.3i)} and {\bf (H1.3ii)} extend immediately to $H\in Vect({\cal H})$ by linearity of $H \to {\cal G}^N_x(H)$ 
for any $x\in {\cal X}$ and $N\geq 1$. 
 \begin{lem}
 \label{extension.e} For any $x\in \mathcal X$, there exists a linear extension $H\in C^2_{b,0} \rightarrow {\cal G}_{x}(H)$ of $ {\cal G}_{x}$ 
such that ${\bf (H1.3)}$ hold for any $H \in  C^2_{b,0}$. 
 \end{lem}
 
As a consequence, writing $\alpha(H) =  \sup_{N,x\in \mathcal X}| {\cal G}_{x}^N (H)|$, for any  $H \in  C^2_{b,0}$,
 \be
\label{alpha}\  \sup_{x\in \mathcal X}| {\cal G}_{x} (H)|\leq  \sup_{N,x\in \mathcal X}| {\cal G}_{x}^N (H)|=\alpha(H)  <+\infty.\ee

\begin{proof} Using \eqref{decomp} and linearity, we only have to prove the extension to ${\cal R}_b$. Let us first prove the result for the compactly supported functions of  ${\cal R}_{b}$.  We consider $H\in {\cal R}_b$ with compact support
and  show that the sequence $({\cal G}_{x}^N (H))_{N}$ converges when $N$ tends to infinity. The function
 $H/|h|^2$ defined on $\mathcal U^*$ can be extended to a continuous  function  $g$ on $\mathcal U$ with compact support and $g(0)=0$. Then by ${\bf (H1.2)}$, there exists a sequence $(g_{n})_n$ of functions of $C^2_{b,0}$ uniformly converging to $g$ and such that  $H_{n}=|h|^2 g_{n} \in Vect({\cal H})$. Since $H=|h|^2g$,   $\| H_{n}-H\|_{h}\rightarrow 0$  when $n\rightarrow\infty$. Moreover  the sequence $\big({\cal G}^N_{.}(H_n)\big)_N$ converges to ${\cal G}_{.}(H_n)$ when $N$ tends to infinity for any fixed $n$ and uniformly ion ${\cal X}$. Let us now consider two integers $m$ and $n$.  Equation \eqref{majj}  tells us that
 $$\sup_{N,x}\big|{\cal G}^N_{x}(H_{m}) - {\cal G}^N_{x}(H_{n})\big| \leq 
\, \alpha(|h|^2)\,\|H_{m}-H_{n}\|_{h}$$
and 
 letting $N$ go to infinity, we obtain that
$({\cal G}_{x}(H_n))_{n}$ is a Cauchy sequence. Then it  converges to a limit denoted by ${\cal G}_{x}(H)$, which satisfies  $\sup_{{\cal X}} |{\cal G}_{.} (H)| <\infty$. Moreover
$$|{\cal G}^N_{x}(H) - {\cal G}_{x}(H)| \leq \big| {\cal G}^N_{x}(H) - {\cal G}^N_{x}(H_{n}) \big| + \big|{\cal G}^N_{x}(H_{n}) - {\cal G}_{x}(H_{n}) \big|+\big|{\cal G}_{x}(H_{n}) - {\cal G}_{x}(H) \big|.$$
Since $\big| {\cal G}^N_{x}(H) - {\cal G}^N_{x}(H_{n}) \big| \leq   \alpha(|h|^2)\, \|g - g_{n}\|_{\infty}$, an appropriate choice of 
  $n$ and  then of $N$ allows us   to upper bound the left hand side by any $\epsilon>0$ and  this ensures
$$\sup_{x\in {\cal X}} |{\cal G}^N_{x}(H) - {\cal G}_{x}(H)| \stackrel{N\rightarrow \infty}{\longrightarrow} 0.$$ 
Let us now consider $H\in \mathcal R_b$. We  introduce a non-decreasing  sequence 
$(\varphi_{n})_{n}\in C^2(\R^d, [0,1])$ such that 
$$
\varphi_{n}(x) = \begin{cases}
  1 & \text{on $B(0,n)$ and} \\
  0 & \text{ on $B(0,n+1)^c$}.
 \end{cases}
 $$
 For $x\in {\cal X}$ and $N\geq 1$, 
$$
|{\cal G}^N_{x}(H \varphi_{m}) - {\cal G}^N_{x}(H \varphi_{n})| \leq \|H\|_{\infty} \, {\cal G}_{x}^N(\un_{B(0,n)^c}) 
\leq \|H\|_{\infty}\,C_{n},\qquad  m\geq n\geq N\geq1 $$
where $C_{n
}\rightarrow 0$  as $n\rightarrow \infty$ by  {\bf (H0)}. Letting $N$ tend to infinity, we obtain that  for any $x\in {\cal X}$, the sequence $({\cal G}_{x}(H \varphi_{n}))_{n}$ is Cauchy and converges  to some real number  ${\cal G}_{x}(H )$.
Moreover $ |{\cal G}_{x}(H)- {\cal G}_{x}(H \varphi_{n})|\leq C_n \|H\|_{\infty}$. It  follows that for any $H \in  {\cal R}_{b}$, 
\ben
 |{\cal G}^N_{x}(H)- {\cal G}_{x}(H)|&\leq&  |{\cal G}^N_{x}(H)- {\cal G}^N_{x}(H \varphi_{n})| + |{\cal G}^N_{x}(H \varphi_{n})- {\cal G}_{x}(H\varphi_{n})|
 + |{\cal G}_{x}(H \varphi_{n})- {\cal G}_{x}(H)|\\
 &\leq & 2C_n  \|H\|_{\infty}+  |{\cal G}^N_{x}(H \varphi_{n})- {\cal G}_{x}(H\varphi_{n})|
 \een
As $H\varphi_{n}\in {\cal R}_b$ and has compact support, ${\cal G}^N_{.}(H \varphi_{n})- {\cal G}_{.}(H\varphi_{n})$ and then 
  ${\cal G}^N_{.}(H)- {\cal G}_{.}(H)$
  tend to $0$ as $N$ tends to infinity uniformly on ${\cal X}$. 
It proves ${\bf (H1.3)}$ and \eqref{alpha}.
\end{proof}

\me We now prove  that a $\sigma$-finite measure can be associated to ${\cal G}_x$ for each $x\in {\cal X}$. It  describes the jumps of the limiting process.
\begin{lem}
\label{extension.mes}   
There exists a  family of $\sigma$-finite measures $ (\mu_{x} : x\in \mathcal X)$ on ${\cal U}^*$ such that   for any $x\in {\cal X}$ and  $H \in \mathcal{R}_{b}$,
 \be
 \label{mesidt}
 {\cal G}_{x}(H) = \int_{\cal U^*} H(u) \mu_{x}(du).
 \ee
For any $x\in {\cal X}$,  ${\cal G}_{x}$ is then extended by \eqref{mesidt} to  any measurable and bounded function $H$ on  $(\R^d)^*$ such that $H(u)=o(\vert u\vert ^2)$. Moreover
 \be
 \label{queueG}
\lim_{b\rightarrow\infty} \sup_{x\in {\cal X}}  \vert {\cal G}_{x}(\un_{B(0,b)^c})\vert = 0.
\ee
\end{lem}
\begin{proof} For any $x\in \R^d$ and $H\in C_{c}({\cal U}^*, \R)$, the map $H \rightarrow {\cal G}_{x}(H)$ is a positive linear operator. Adding that
 ${\cal U}^*$ is locally compact, Riesz Theorem  leads to the existence of a  $\sigma$-finite measure $\mu_{x}$ on ${\cal U}^*$ such that for any $H\in C_{c}({\cal U}^*, \R)$, 
${\cal G}_{x}(H)= \int_{{\cal U}^*} H(u) \mu_{x}(du).$ 
The extension of this identity to any $H\in \mathcal R_b$ follows again from  an approximation procedure, using
 $\varphi_{n}$  defined in the proof of Lemma \ref{extension.e}. Indeed, on  the one hand monotone convergence ensures that  
 $\int_{{\cal U}^*} H \varphi_{n} \mu_{x}$ goes to  $ \int_{{\cal U}^*} H \mu_{x}$.
On the other hand,  
   $\vert {\cal G}_{x}(H \varphi_n)- {\cal G}_{x}(H)\vert \leq  C_n\|H\|_{\infty}$ goes to $0$.  
Finally \eqref{queueG}  comes from {\bf (H0)} with a monotone approximation of  $\,\un_{B(0,b)^c}$ by elements of ${\cal R}_{b}$  and the convergence of ${\cal G}^N$ to ${\cal G}$.  
\end{proof}



We now prove the convergence of conditional increments functionals, defined for  any function $H\in C^{0}_{b,2}$  and  $t>0$
by 
\be
\label{increment}
\phi^N_{t}(H) =  \sum_{k=1}^{[v_{N}t]} \E\big(H(X^N_{k} - X^N_{k-1})\,|\, {\cal F}^N_{k-1} \big) =  \,{1\over v_{N}} \sum_{k=1}^{[v_{N}t]} {\cal G}^N_{X^N_{k-1}}(H),
\ee
where the last identity follows from the Markov property.

\begin{prop}
\label{key-prop}
For any function $H\in C^{0}_{b,2}$  and $t>0$, 
$$
\lim_{N\to \infty}\sup_{t\leq T}\Big|\phi^N_{t}(H)- \int_{0}^t {\cal G}_{X^N_{[v_{N}s]}}(H) \,ds\Big|  = 0 \qquad \text{a.s}.
$$
\end{prop}
 
 \begin{proof} Using \eqref{alpha},  we have
$${1\over v_{N}} \sum_{k=1}^{[v_{N}t]} {\cal G}^N_{X^N_{k-1}}(H)  = \int_{0}^t {\cal G}^N_{X^N_{[v_{N}s]}}(H)\, ds - \int_{\frac{[v_{N} t]}{v_{N}}}^t {\cal G}^N_{X^N_{[v_{N}s]}}(H)\, ds=\int_{0}^t {\cal G}^N_{X^N_{[v_{N}s]}}(H)\, ds + {\cal O}\left(\alpha(H)\over v_{N}\right).$$
\me 
Then
$$
\sup_{t\leq T}\Big|\phi^N_{t}(H)- \int_{0}^t {\cal G}_{X^N_{[v_{N}s]}}(H) ds\Big|  \leq T\, \sup_{x\in {\cal X}}|{\cal G}^N_{x}(H)-{\cal G}_{x}(H)|+  {\cal O}\left(\alpha(H)\over v_{N}\right)\label{control}
$$
and the conclusion follows from ${\bf (H1.3i})$, which holds for $H$ thanks to    Lemma \ref{extension.e}. 
\end{proof}

\bi We define on the canonical space $\mathbb{D}([0,\infty),  {\cal X})$  a triplet which   characterizes  the limiting values of the sequence $\,(X^N_{[v_{N}.]}, N\in \mathbb{N})$. Using the measurability and boundedness of 
$x\rightarrow {\cal G}_{x}(f)$  for $x\in {\cal X}$ and  $f\in C^2_{b,0}$ and  the truncation function $h_{0}$  introduced in \eqref{troncationjacod},  we define for any $\omega=(\omega_s, s\geq 0) \in \mathbb D([0,\infty), {\cal X})$ the functionals
    \begin{equation}\label{carac}
 \left.
 \begin{aligned}
 B_t(\omega)&=\int_0^t \Big({\cal G}_{\omega_{s}}(h^1_{0}), \cdots, {\cal G}_{\omega_{s}}(h^d_{0})\Big) \,ds,\\
\widetilde  C_t^{ij}(\omega)&=\int_0^t {\cal G}_{\omega_{s}}(h^i_{0} h^j_{0})\, ds, \\
\nu_t(\omega,H)&=\int_0^t  {\cal G}_{\omega_{s}}(H \un_{{\cal U}})\,ds =\int_0^t \int_{{\cal U}^*} H(u) \mu_{\omega_{s}}(du) \,ds
\end{aligned}
\right\} 
\end{equation}
for any $H\in C_b(\R^d,\R)$ such that $H(u)=o(\vert u\vert^2)$. The last identity comes from \eqref{mesidt}.
  
 \me As in Chapters II.  2 \& 3 in  \cite{Jacod} adapted to the state space $\overline{\cal X}$ (instead of $\mathbb{R}^d$), the characteristic triplet associated with the semimartingale $X^N$ is
 given  for $i,j\in \{1, \ldots,d\}$ by
    \begin{equation}
 \left.
 \begin{aligned}
 B^N_{t}&=\sum_{k\leq [v_{N}t]} \E(h_0(U^N_{k})| {\cal F}^N_{k-1})=(\phi^N_{t}(h^1_{0}),\cdots, \phi^N_{t}(h^{d}_{0}))\\
\widetilde C^{N,ij}_{t}&= \sum_{k\leq [v_{N}t]} \Big(\E(h^i_{0}(U^N_{k}) h^j_{0}(U^N_{k})| {\cal F}^N_{k-1})  - \E(h^i_{0}(U^N_{k}) | {\cal F}^N_{k-1})  \E(h^j_{0}(U^N_{k}) | {\cal F}^N_{k-1}) \Big) \\
\phi^N_{t}(H) &=  \sum_{k\leq [v_{N}t]} \E(H(U^N_{k})  | {\cal F}^N_{k-1}),
\end{aligned}
\right\} \label{caractlim}
\end{equation}
where $U^N_{k}=X^N_{k} - X^N_{k-1}$ and $H$ is a continuous bounded function on $\R^d$ vanishing in a neighborhood of $0$.
Proposition \ref{key-prop} implies the convergence of the characteristics, as stated in the next proposition.
  \begin{prop}
\label{prop-convergenceh}
  For any $T>0$ and any $i,j =1,\cdots, d$
  and any $H\in C_b({\cal U},\R)$
 equal to $0$  in some  neighborhood of $0$, we have the following almost-sure convergences
 
  \be
 && \sup_{t\leq T}\, \left|\   B^{N,i}_{t} - B^i_{t}\circ X^N_{[v_{N}.]}\right| \stackrel{N\to \infty}{\longrightarrow} 0 ;\label{condB}\\
 && \sup_{t\leq T}\, \left| \widetilde C^{N,ij}_{t}- \widetilde C^{ij}_{t}\circ X^N_{[v_{N}.]}\right|  \stackrel{N\to \infty}{\longrightarrow} 0  ;
  \label{condC}\\
 && \sup_{t\leq T}\, \left|\phi^N_{t}(H) - \nu_{t}(X^N_{[v_{N}.]}, H)\right|   \stackrel{N\to \infty}{\longrightarrow} 0.\label{condNu}\ee
\end{prop}

\me
\begin{proof}  From Proposition \ref{key-prop}, we immediately obtain  the first and last convergences and
 $$\sup_{t\leq T}\, \left|\phi^N_{t}(h^i_{0}h^j_{0}) - \widetilde C^{ij}_{t}\circ X^N_{[v_{N}.]}\right| \rightarrow_{N\to \infty} 0 \quad \text{a.s}.$$
 So it remains to be replaced $\phi^N_{t}(h^i_{0}h^j_{0})$ by $\widetilde C^{N,ij}_{t}$. We have
$$| \E(h^k_{0}(U^N_{k}) | {\cal F}^N_{k-1}) | \leq {1\over v_{N} } \sup_{N, x\in {\cal X}}  \big\vert {\cal G}^N_{x}(h_0^k) \big\vert  \leq  \alpha(h_0^k)  .{1\over v_{N} }$$ 
and  $\alpha(h_0^k)<\infty$ from \eqref{alpha}.  Hence the second term in $\widetilde C^{N,ij}_{t}$ tends to $0$ as $N \to \infty$, which yields the result.
\end{proof}

We are now in position to provide a proof of Theorem \ref{thm-tightness}. In order to apply
 Theorem 3.9 IX  p543  in \cite{Jacod} and get the tightness, we need  to check the  strong majoration hypothesis
and the condition on big jumps required in its statement.

 First, 
if $H\in C^0_{b,2}$ then $\mathcal G_. H$ is bounded and
there exists a positive constant $A$ such that for any $\omega \in \mathbb D([0,\infty), {\cal X})$,

\be
\label{domm}
\sum_{i=1}^d \Var(B^i(\omega))_{t}  + \sum_{i,j=1}^d \tilde C^{ij}_t(\omega)  + \nu_t(w, r)\leq A \, t,
\ee
where  $\,\Var(X)_{t}$ denotes the total variation of $X$  on $[0,t]$ and $r(u):=| u|^2 \wedge 1$.
 
Second, to  control the big jumps, we use  the fact that 
$\nu_{t}(.,\un_{B(0,b)^c}) \leq t\, \|{\cal G}_{.}(\un_{B(0,b)^c})\|_{\infty}$,
which tends to $0$ as $b$ tends to infinity from \eqref{queueG}.  We thus obtain
\be
\label{biggj}
\lim_{b \uparrow \infty} \sup_{w\in  \mathbb D([0,\infty), {\cal X})}  \nu_t(w, \un_{B(0,b)^c} )=0.
\ee

The tightness of  $\,(X^N_{[v_{N}.]}, N\in \mathbb{N})$  follows from \eqref{condB}-\eqref{biggj} and from the tightness of the initial condition, by an application of    the mentioned theorem in \cite{Jacod}.   
\subsubsection{Proofs of Theorems \ref{identification} and \ref{main}}

\bi Let us now assume the additional Hypothesis  {\bf (H2)}. We wish to identify the limiting values of $\,(X^N_{[v_{N}.]}, N\in \mathbb{N})$ as solutions of the stochastic differential system  \eqref{eds}. We first need to continuously extend   the limiting characteristic triplet to the boundary.
\begin{lem}
(i)  For any  $H\in {\cal R}_b$, the map $  x\in  {\cal X} \rightarrow{\cal G}_{x}(H) $ is continuous and extendable by continuity to $ \overline {\cal X}.$ Moreover
\be
\label{max}
\sup_{x\in \overline{\cal X}}  | {\cal G}_{x} (H)|\leq \alpha(H)  <+\infty.
\ee
(ii)  For any  $H\in {\cal R}_b$ and $x\in \overline{{\cal X}}$, 
\be
\label{idmeas}
{\cal G}_{x}(H)=\int_{V} H(K(x,v)) \mu(dv)
\ee
and $ \int_{V}   1\wedge\vert K(.,v)\vert^2  \mu(dv)$ is bounded on $\overline{\cal X}$.
\end{lem}
\begin{proof} Let $H \in {\cal R}_b$. Using the sequences $\varphi_n$  and $(H_n)_n$ defined in the proof of Lemma \ref{extension.e} and approximating $\varphi_nH$ for $\parallel . \parallel_h$ by $H_n \in Vect({\cal H}) \cap {\cal R}_b$ as in the proof of Lemma \ref{extension.mes},  we obtain 
$$\sup_{x \in {\cal X}, N\geq 1}  \big\vert {\cal G}_x^N(H)-{\cal G}^N_x(H_n)\big\vert \leq \parallel H \|_{\infty}C_n + \parallel \varphi_n H-H_n \parallel_h \alpha(\vert h\vert^2),$$
which tends to $0$ as $n\rightarrow \infty$. Letting $N\rightarrow \infty$ ensures that ${\cal G}_.H_n$ converges uniformly to ${\cal G}_.H$ as $n\rightarrow \infty$.  
Combining  this with  ${\bf (H2.1)}$ applied to $H_n$, we deduce that ${\cal G}_.H$ is continuous on ${\cal X}$ and extendable by continuity to $\overline{\cal X}$. Moreover  \eqref{alpha}  yields \eqref{max} by continuity, which proves $(i)$. 

For $(ii)$, we first  consider $H\in Vect( {\cal H})\cap {\cal R}_b$. Then  $\alpha^{h_0}(H)=\beta^{h_0}(H)=0$, $\overline{H}^{h_0}=H$ and  ${\bf (H2.2)}$ 
 ensures that \eqref{idmeas}  holds for $H$. Let us now extend this identity to $H\in {\cal R}_{b}$ with compact support. We note that  $H=|h|^2 g$ with $g\in C_{c}({\cal U}, \mathbb{R})$. By ${\bf (H1.2)}$,  the function $g$ is uniformly approximated by a sequence $g_{n}$ such that 
$|h|^2 g_{n} \in Vect({\cal H}) \cap {\cal R}_b$. The identity \eqref{h2} implies that \eqref{idmeas} holds for any $|h|^2 g_{n}$ and
$$\forall x\in \overline{\cal X}, \quad {\cal G}_{x}(|h|^2 g_{n})=\int_{V} (|h|^2 g_{n})(K(x,v)) \mu(dv).$$
We let  $n$ tend to infinity in both terms using  \eqref{alpha} and  the assumption $ \int_{V}   1\wedge\vert K(x,v)\vert^2  \mu(dv)<+\infty$. The extension to ${\cal R}_b$ follows  again from a monotone approximation by the compactly  supported functions $H\varphi_n$, which ends the proof.
\end{proof}
This lemma allows us to extend the definitions of the characteristics and the identities of \eqref{carac} to any $w\in \mathbb D([0,\infty), \overline{\cal X})$. Moreover $(i)$ ensures that    $w\in \mathbb D([0,\infty), \overline{\cal X})\rightarrow (B_t(\omega),
\widetilde  C_t(\omega),
  \nu_t(\omega,H))$ is continuous and  that the dominations 
\eqref{domm} and \eqref{biggj} extend  from ${\cal X}$ to $\overline{{\cal X}}$.
 We can then apply
 \cite[Theorem 2.11, chapter IX, p530]{Jacod} on the closed set $\overline {\cal X}$ for the identification.
We obtain that
any limiting value of the law of $(X^N_{[v_N.]})_N$ is a  solution   of the martingale problem  on the canonical space $ \mathbb D([0,\infty), \overline{\cal X})$ with characteristic triplet $(B,C,\nu)$, where $$C_t^{ij}=\widetilde{C}_t^{ij}-\nu_t(.,h_0^ih_0^j).$$

Finally, using {\bf (H2.2)} for $H\in \{h_0^i,h_0^{i}h_0^{j}\}$ and \eqref{idmeas}, the characteristics   in \eqref{carac} can be written as
\ben
B_t(w)&=&\int_0^tb(w_s)ds \nonumber \\
C_t^{ij}(w)&=&\int_0^t \left(\sum_{k=1}^{d} 	\sigma_{i,k}(w_s)\sigma_{j,k}(w_s) \right)ds \nonumber \\
\nu_t(w,H) &=& \int_0^t \int_{V} H(K(w_s,v)) \mu(dv)ds, 
\een
for  any $w \in \mathbb D([0,\infty), \overline{\cal X})$.
By \cite[Chapter III,     Theorem 2.26 p157]{Jacod}, the set of solutions of  the martingale problem  with characteristic triplet $(B,C,\nu)$ coincides with the set of weak solutions of the stochastic differential equation \eqref{eds}.   The proof of Theorem \ref{identification} is now complete. \\
 
  \bi To conclude the proof of  the convergence, we  remark that \bl{the} uniqueness hypothesis ({\bf H3}) guarantees   $(iii)$ in \cite{Jacod} Theorem 3.21, chapter IX, p.546]. The other points $(i-vi)$ of this theorem have been checked above and Theorem \ref{main} follows.



\section{Wright-Fisher process with selection in  L\'evy environment}
\label{WF}
\subsection{The discrete model}
Let us consider the framework of the Wright-Fisher model: at each generation, the alleles of  a fixed size population are sampled from the previous generation. We consider a population of $N$ individuals characterized by some allele.  The number of individuals carrying this allele is a  process
$\,(Z^N_{k}, k\in \mathbb{N})\,$ whose dynamics depends on the environment.
When $N\geq 1$ is fixed, we consider the coupled process describing the discrete time dynamics of the population process and the environment process. 
 It is recursively defined for $k\geq 0$ by
\be \begin{cases}
Z^N_{k+1} =  \sum_{i=1}^N {\cal E}^N_{k,i}(Z^N_{k}/ N, E^N_{k}),\\
\smallskip \\
S^N_{k+1} = S^N_{k}+ E^N_{k}, \label{def-env}
\end{cases}
\ee
and $\,S^N_0 = 0, Z^N_0 =[NZ_0]$, $Z_0\in [0,1]$ is a finite random variable, $(E^N_{k})_{k}$ are independent and identically distributed with values in $(-1, +\infty)$ and the  family of 
random variables $\,\big( ({\cal E}^N_{k,i}(z,w), (z,w)\in [0,1]\times (-1,\infty)); k\ge 1, i\geq 1\big)$ 
 are independent. Moreover for each $(z,w)\in [0,1]\times (-1,\infty)$,
  the random variables $({\cal E}^N_{k,i}(z,w);  k\ge 1, i  \geq 1\big)$ are identically distributed as a  Bernoulli random variable ${\cal E}^N(z,w)$ defined by
$$\p({\cal E}^N(z,w) = 1) = p(z,w)\ ;\ \p({\cal E}^N(z,w) = 0) = 1 - p(z,w).$$
We also assume that $Z_0$,  $\big( ({\cal E}^N_{k,i}(z,w), (z,w)\in [0,1]\times (-1,\infty)); k\ge 1, i\geq 1\big)$ and  $(E^N_{k}, k\geq 0)$ are independent.\\
Moreover $p$ is a $C^3$-function   from $[0,1]\times (-1,\infty)$ to $[0,1]$ verifying  $p(z,0) = z$ for any $z\in [0,1]$. 
A main example,   developed in Section \ref{ex.p},
is given by $p(z,w) = z(1+w)/(z(1+w) + 1-z)$ and extends the classical Wright Fisher model  with rare selection to random environments.\\

Following  \cite{Jacod} [chap.VII Corollary 3.6,p.415], we state an assumption for the random walk  $S^N_{[N.]}$ to converge in  law to a L\'evy process with characteristics $(\alpha_{\e}, \beta_{\e}, \nu_{\e})$.
Let us 
consider  a truncation function $h_{\e}$ defined on $(-1,+\infty)$,  i.e.  continuous and  bounded and  satisfying  $h_{\e}(w)=w$ in a neighborhood of $0$. 
For convenience, we also assume that  $h_{\e}(w)\ne 0$ for any $w\ne 0$. 

\begin{assumptionA} 
There exist $\alpha_{\e}\in \mathbb{R}$, $\sigma_{\e}\ge 0$ and a measure $\nu_{\e}$ on $(-1, +\infty)$ satisfying $\int_{(-1,+\infty)}(w^2\wedge 1) \nu_{\e}(dw) < +\infty$ such that 
$$  \lim_{N\to \infty } N\, \E(h_{\e}(E^N))=\alpha_{\e} \ ;\ \lim_{N\to \infty } N \,\E(h_{\e}^2(E^N))=\beta_{\e} = \sigma_{\e}^2 + \int_{(-1,\infty)} h_{\e}^2(w)\nu_{\e}(dw), $$
$$ \lim_{N\to \infty } N\, \E(f(E^N))=  \int_{(-1,\infty)} f(w) \nu_{\e}(dw),$$
for any $f$ vanishing in a neighborhood of $0$, continuous and    bounded. 
\end{assumptionA}

\me The small fluctuations of the environment are given by $\sigma_{\e}$, while the dramatic events are given by the jump measure $\nu_{\e}$. Negative jumps will correspond to dramatic disadvantages of allele $A$ and an usual set of selection coefficient  is   $(-1,\infty)$, as illustrated  in Section \ref{ex.p}. \\
The limiting 
 environment  process $Y$ can thus be   defined by   \be
\label{environnement}
Y_t = \alpha_{\e} t + \int_0^t \sigma_{\e} dB^{\e}_s +  \int_0^t\int_{(-1,+\infty)}  h_{\e}(w) \widetilde N^{\e}(ds, dw) +  \int_0^t \int_{(-1,+\infty)}(w - h_{\e}(w))  N^{\e}(ds, dw),
\ee
where $B^{\e}$ is a Brownian motion and $N^{\e}$ is a Poisson point measure on $\mathbb{R}_{+}\times (-1,+\infty)$ independent of $B^{\e}$  with intensity measure $\nu_{\e}$. By construction, this L\'evy process has jumps larger than $-1$. \\

\me
 Let us  first prove a consequence  of Assumption {\bf A} which will be needed in  the proof of the next theorem.
  \begin{lem}
\label{CV-EN}   Let $g \in C^3([0,1]\times (-1,\infty), \R)$ bounded and satisfying $g(z,0)=0$ for any $z\in[0,1]$. Then, under Assumption {\bf A},
$$N\E(g(z,E^N)) \stackrel{N\rightarrow \infty}{\longrightarrow} {\cal B}_z (g), $$
uniformly for $z\in [0,1]$, with
$${\cal B}_z (g)=\alpha_{\e} \frac{\partial g}{\partial w} (z,0) + {\beta_{\e} \over 2} \frac{\partial^2 g}{\partial w^2} (z,0)+\int_{(-1,\infty)}
\widehat{g}(z,w)\nu_{\e}(dw)$$
and 
 $\widehat{g}(z,w)=g(z,w)-h_{\e}(w)\frac{\partial g}{\partial w} (z,0)-{h_{\e}(w)^2 \over 2} \frac{\partial^2 g}{\partial w^2} (z,0)$.
\end{lem}

\begin{proof} Indeed, we can decompose $N\E(g(z,E^N))$ as follows
$$N\E(g(z,E^N))=\frac{\partial g}{\partial w} (z,0)\, N\E(h_{\e}(E^N))+{1 \over 2} \frac{\partial^2 g}{\partial w^2} (z,0)\, N\E(h_{\e}(E^N)^2)+N\E(\widehat{g}(z,E^N)).$$
The  first two
 terms converge  uniformly as $N\rightarrow\infty$ by a direct application of Assumption {\bf A}. Moreover
the last part of Assumption {\bf A} can be extended to  any continuous function $f(w)=o(w^2)$ using a monotone approximation of  $f$ by functions vanishing in a neighborhood of $0$. Then the last term converges  for  fixed $z$ and 
it remains to prove that the convergence is  uniform on $[0,1]$.  First, let us consider a compact subset $K=[0,1]\times[-1+\varepsilon_0, A]$ of  $[0,1]\times(-1,\infty)$. As
$g$ is $C^3([0,1]\times (-1,\infty), \R)$, the function
$$(z,w) \rightarrow \frac{\widehat{g}(z,w)}{h_{\e}(w)^2}=\frac{g(z,w)-h_{\e}(w)\frac{\partial g}{\partial w} (z,0)}{h_{\e}(w)^2}-{1 \over 2} \frac{\partial^2 g}{\partial w^2} (z,0)$$
and its first derivative with respect to $z$ are well defined  on $[0,1]\times (-1,\infty)\setminus[0,1]\times \{0\}$ and extendable by continuity to $[0,1]\times (-1,\infty)$. Thus
the derivative of $\widehat{g}(z,w)/h_{\e}(w)^2$ with respect to $z$ is bounded on $K$. As $(N\E(h_{\e}(E^N)^2))_N$ is bounded by the second part of Assumption {\bf A},
there exists $C>0$ such that for any $N\geq 1$,
$$\bigg\vert N\E\left(\widehat{g}(z,E^N)1_{E^N\in [-1+\varepsilon_0, A]} \right)-N\E\left(\widehat{g}(z',E^N)1_{E^N\in [-1+\varepsilon_0, A]}\right)\bigg\vert \leq C\vert z-z'\vert.$$
Moreover, since all functions  involved in the definition of $\widehat{g}$ are bounded,  there exists $C'>0$ such that
$$\bigg\vert N\,\E\left (\vert \widehat{g}(z,E^N)\vert 1_{E^N\not\in [-1+\varepsilon_0, A]} \right)\bigg\vert \leq C'N\,\mathbb P(E^N\not\in [-1+\varepsilon_0, A])$$
and by the last part of Assumption {\bf A},
$$\lim_{\varepsilon_0\rightarrow 0, A\rightarrow \infty} \sup_N N\,\mathbb P(E^N\not\in [-1+\varepsilon_0, A])=\lim_{\varepsilon_0\rightarrow 0, A\rightarrow \infty} \nu_{\e}((-1,-1+\varepsilon_0)\cup(A,\infty))=0.$$
Combining the  last two inequalities, we obtain that  the family of functions $(N\,\E(\widehat{g}(.,E^N)))_N$ is uniformly equicontinuous on $[0,1]$ and the convergence  is uniform by Ascoli Theorem. 
\end{proof}
 

We can now generalize the classical convergence in law to the Wright-Fisher diffusion with selection to i.i.d. environments.

\subsection{Tightness and identification}

We are interested in the asymptotic behavior of the Markov chain  $$X^N_k=\left({Z^{N}_{k}\over N},  S^N_{k}\right),\quad k\in \mathbb N$$ when $N$ tends to infinity. This process 
 takes values in $\mathcal X=[0,1]\times \R$. 
 
\me For the statement, we introduce the drift
coefficient inherited from the fluctuations of the environment:
$$b_1(z)=\alpha_{\e}{\partial p \over \partial w} (z,0)+{\sigma_{\e}\over 2}{\partial^2 p \over \partial w^2 } (z,0)+\int_{(-1,\infty)} \left(p(z,w)-z-  h_{\e}(w)\frac{\partial p}{\partial w} (z,0)\right)\nu_{\e}(dw).$$ 

\begin{thm}
\label{WF-tension}
Under Assumption ${\bf A}$, the sequence of processes $\,\left(Z^{N}_{[N.]}/N,  S^N_{[N.]}\right)_N\,$ is tight   in $\mathbb D([0,\infty),[0,1]\times \R)$
and any  limiting value of this sequence is  solution of the following stochastic differential equation
\be
\begin{split} \label{EDSS}
 Z_t &=Z_0+\int_0^t b_{1}(Z_s)ds + \int_0^t \sqrt{Z_s(1-Z_s)} dB_s^{\D} + \sigma_{\e}\int_0^t {\partial p \over \partial w} (Z_s,0) \,d B^{\e}_s  \\
& \qquad  +\int_{(-1,\infty)} (p(Z_{t-},w) - Z_{t-}) \widetilde N(dt,dw) ;\\
\ 
 Y_{t} &= \alpha_{\e} t + \sigma_{\e} B^{\e}_{t}  + \int_0^t\int_{(-1,\infty)} h_{\e}( w) \widetilde N(dt,dw) +\int_0^t \int_{(-1,\infty)} (w -h_{\e}(w) ) N(dt,dw),
\end{split}
\ee
where $B^{\D}$ and $B^{\e}$ are Brownian motions; $N$ is a  Poisson point measure on $\mathbb{R}_{+}\times (-1,\infty)$ with intensity $\, dt\nu_{\e}(dw)\,$ and $\widetilde N$ is the compensated martingale measure of $N$; $Z_0, B^{\D}, B^{\e}$ and  $N$ are   independent.    
\end{thm}

\begin{proof} We apply our results to the  Markov chain
$\ X^N_k=\left(\big({Z^{N}_{k}\over N},  S^N_{k}\big), k\in \mathbb N\right)\,$. \\
Let $x=(z,y)\in \mathcal X$, we set
$F^N_{x} = F^N_{(z,y)} = \left({1\over N} \sum_{i=1}^{N}{\cal E}_{i}(z, E^N), y + E^N \right)$
and we have
\be
\label{fnx}
F^N_{x} -x
=  \left({1\over N} \sum_{i=1}^N({\cal E}_{i}(z, E^N) - z), E^N\right) .\ee
The state space of the random variables $F^N_{x} -x$  is $\,{\cal U} = [-1,1]\times (-1,+\infty)$.

\bi We first prove that  {\bf (H0)}, {\bf (H1)} and  {\bf (H2)} are satisfied with $v_N=N$.

\me (i) Let us first check {\bf (H0)}. We take $b>0$ and consider
$$ {\cal G}^N_x (\un_{{\cal B}(0,b)^c}) = N\, \E(\un_{{\cal B}(0,b)^c}(F^N_{x} - x)).$$
Then 
\ben N\,\E(\un_{{\cal B}(0,b)^c}(F^N_{x} - x)) &\le & N\,\mathbb{P}\left({1\over N} \big|\sum_{i=1}^N ({\cal E}_{i}(z, E^N) - z)\big|>b/\sqrt{2}\right) + 
N\,\mathbb{P}\left( |E^N| >b/\sqrt{2}\right). \een
We observe that ${1\over N} \big|\sum_{i=1}^N ({\cal E}_{i}(z, E^N) - z)\big| \leq 1$ 
 a.s. Moreover the last part of Assumption {\bf A} ensures that
$$\limsup_{N\rightarrow \infty} N\,\mathbb{P}( |E^N| >b/\sqrt{2}) \leq \nu[b/\sqrt{2}-1, \infty),$$
which tends to $0$ as $b\rightarrow +\infty$.
 Then $\sup_{N, x\in [0,1]\times (-1,\infty)}{\cal G}^N_x (\un_{{\cal B}(0,b)^c})$ tends to $0$ and  {\bf (H0)} is satisfied. 

 \bi (ii)  We define the function $h$ on ${\cal U} $ by  $$h(u,w)= (1-e^{-u}, 1-e^{-w}).$$
 The space $\mathcal H$  is the subset of  real functions on ${\cal U}$   defined as
$$\mathcal H=\{ (u,w) \in{\cal U} \rightarrow  H_{k,\ell}(u,w) \, ; \,  k,   \ell \geq 0\}, \quad \text{with } \ H_{k,\ell}(u,w) = 1 - e^{- k u -\ell w}.$$
We can apply the local Stone-Weierstrass Theorem  to the algebra $Vect({\cal H})\cap  C_0({\cal U}^*)$, ${\cal U}^*={\cal U}\setminus\{0,0\}$ being a  locally compact Hausdorff space   (see Appendix \ref{localSW}). This algebra in dense in $C_0({\cal U}^*)$ and then any function in $C_c({\cal U})$ vanishing at zero is the uniform limit of elements of
$Vect({\cal H})$. Moreover $Vect({\cal H})$ is stable by multiplication by $\vert h \vert^2$. We deduce that ${\bf (H1.2)}$ is satisfied, while  ${\bf (H1.1)}$ is obvious.

\me Let us now prove that ${\bf (H1.3)}$ is  satisfied. We need to study the limit of ${\cal G}^N_x(H_{k,\ell})$ as $N$ tends to infinity. Recall that ${\cal G}^N_x (H_{k,\ell})  = N\, \E(H_{k,\ell}(F^N_{x} - x))$
with $x=(z,y)$ and $F^N_{x} - x$ given by \eqref{fnx}.
We have
\ben
{\cal G}^N_x (H_{k,\ell}) &=& {N}\,\E\Big(1- e^{-{k\over N}\sum_{i=1}^N ({\cal E}_{i}(z, E^N)-z)}e^{-\ell E^N}\Big)\\
&=&  {N}\Big(1- \E\Big( \E\Big[e^{-{k\over N} ({\cal E}(z, E^N)-z)} \,|\,E^N\Big]^N e^{-\ell E^N}\Big)\Big)\\
&=&  {N}\Big(1- \E\Big( \Big[e^{-{k\over N}(1-z)}p(z, E^N) +e^{{k\over N}z}(1- p(z, E^N) )\Big]^N e^{-\ell E^N}\Big)\Big).
\een
The following Taylor expansion gives
\ben
\log\bigg(e^{-{k\over N}(1-z)}p +e^{{k\over N}z}(1- p)\bigg)
&=& {k\over N}(z-p) +{k^2\over 2 N^2 } p(1-p) +\mathcal O(1/N^3),
\een
with $N^3\mathcal O(1/N^3)$ bounded uniformly in $p,z \in [0,1]$. Then  we obtain
\ben
 {\cal G}^N_x(H_{k,\ell}) &=&N\E\Big(1- e^{  k(z-p(z,E^N))-\ell E^N}.e^{{k^2\over 2N }p(z,E^N)(1-p(z,E^N))}.e^{\mathcal O(1/N^2)}\Big)\\
 &=&N \,\E\Big(1 - \left[(1-A_{k,\ell})(1+B_{k,N})(1+ R_{k,N}) \right] (z,E^N) \Big),
 \een
where $N^2R_{k,N}(z,w)$ is  uniformly bounded for $z \in [0,1], w\in (-1,\infty)$ and $N\geq 1$ and  $$ A_{k,\ell }(z,w)=  1-\exp\big(-k(p(z,w)-z) -\ell w\big);  \quad  B_{k,N}(z,w)= {k^2\over 2 N}\, p(z,w)(1- p(z,w)) +\mathcal O\big( {1 \over N^2} \big).$$
By expansion, we deduce that \be \label{DLWF}
{\cal G}^N_x(H_{k,\ell})
&=&N\E\big(A_{k,\ell}(z,E^N)\big)\Big(1+\mathcal O\big(1 / N \big)\Big) \nonumber \\
&&\qquad \qquad -{k^2 \over 2}\E\Big( p(z,E^N)(1-p(z,E^N))\Big)+\mathcal O(1/N).
\ee

\me Using Lemma \ref{CV-EN} both for $(z,w)\rightarrow A_{k,\ell}(z,w)$ and $(z,w)\rightarrow p(z,w)(1-p(z,w))-z(1-z)$, we obtain from \eqref{DLWF} that
$${\cal G}^N_x(H_{k,\ell}) \stackrel{N\rightarrow \infty}{\longrightarrow} {\cal G}_x(H_{k,\ell})={\cal B}_z(A_{k,\ell})-{k^2\over 2 }z(1-z),$$
uniformly for $x=(z,y)\in [0,1]\times \R$.  Then {\bf (H1.3i)} is satisfied and for any $x \in {\cal X}$,
\be
{\cal G}_x(H_{k,\ell}) =\alpha_{\e} \frac{\partial A_{k,\ell}}{\partial w} (z,0)
+  {\beta_{\e} \over 2}\frac{\partial^2 A_{k,\ell}}{\partial w^2} (z,0)+\int_{(-1,\infty)}  \widehat{A_{k,\ell}}(z,w)\nu_{E}(dw)  \,  -{k^2\over 2 }z(1-z)\label{valGlim}
\ee
with 
$$\frac{\partial A_{k,\ell}}{\partial w} (z,0)= k \frac{\partial p}{\partial w} (z,0)+\ell, \qquad  \frac{\partial^2 A_{k,\ell}}{\partial w^2} (z,0)=k \frac{\partial^2 p}{\partial w^2} (z,0)-\big(k \frac{\partial p}{\partial w} (z,0)+\ell\big)^2$$
and
$$\widehat{A_{k,\ell}}(z,w)=A_{k,\ell}(z,w)-h_{\e}(w)\frac{\partial A_{k,\ell}}{\partial w} (z,0)-{h_{\e}^2(w) \over 2} \frac{\partial^2 A_{k,\ell}}{\partial w^2} (z,0).$$

\me The assumptions on the function $p$ allow us to conclude that  {\bf (H1.3ii)} is also satisfied.

\bi (iii) We now check ${(\bf H2)}$.
The continuity  of ${\cal G}_{.}$ on $[0,1]\times \R$   comes from the    regularity of $p$, from  the integrability  assumption on $\nu_{\e}$ and  from Lebesgue's Theorem (by Assumption {\bf A}).

\me Now, let us introduce the truncation function defined on ${\cal U}=[-1,1]\times (-1,\infty)$ by  $$ h_{0}(u,w)= (u, h_{\e}(w)).$$
With the notation of Section 2, recall that $h^1_{0}(u,w)=u$ and $h^2_{0}(u,w)=h_{\e}(w)$.

\me  With the notation of Lemma \ref{lem-decomp} we have
$$\alpha_{1}^{h_{0}}(H_{k,\ell})=k, \ \alpha_{2}^{h_{0}}(H_{k,\ell})=\ell, \  \beta_{11}^{h_{0}}(H_{k,\ell})=-{k^2 \over 2} , \ \beta_{12}^{h_{0}}(H_{k,\ell})=\beta_{21}^{h_{0}}(H_{k,\ell})=-{k\ell \over 2} , \ \beta_{22}^{h_{0}}(H_{k,\ell})=-{\ell ^2 \over 2}.$$
 Moreover, setting  $K(z,w) = (p(z,w)-z, w)$,  we note that $A_{k,\ell}=H_{k,\ell}\circ K$ and
\ben
\widehat{A_{k,\ell}}(z,w)&=&\overline{H_{k,\ell}}^{h_0}(K(x,w))
+kp_1(z,w)-{ k^2 \over 2}p_2(z,w)-k\ell h_{\e}(w) q(z,w),
\een
with $\,p_1(z,w)=p(z,w)-z-  h_{\e}(w)\frac{\partial p}{\partial w} (z,0)-{ h_{\e}(w)^2 \over 2}\frac{\partial^2 p}{\partial w^2} (z,0)$,\\
$p_2(z,w)=(p(z,w)-z)^2-  h_{\e}^2(w) \Big(\frac{\partial p}{\partial w} (z,0)\Big)^2\,$ and $\,q(z,w)=p(z,w)-z- h_{\e}(w)\frac{\partial p}{\partial w} (z,0).$

\me Now we set
 $V= (-1,\infty)$ and  choose $\mu =\nu_{\e}$ and  for $x=(z,y) \in [0,1]\times \R$, we define  \ben
 &&b_{1}(x) =   b_{1}(z) = \alpha_{\e}   \frac{\partial p}{\partial w} (z,0)+  {\beta_{\e} \over 2}  \frac{\partial^2 p}{\partial w^2} (z,0)+\int_V p_{1}(z,w)\nu_{\e}(dw)  \quad ;\quad  b_{2}(x) = \alpha_{\e}\\
 && \sigma_{1,1}(x) = \sqrt{z(1-z)}\quad ; \quad  \sigma_{2,2}(x)  
 = \sigma_{\e} \quad ; \quad \sigma_{2,1}(x)=0\quad ; \quad  \sigma_{1,2}(x) = \sigma_{\e}  \frac{\partial p}{\partial w} (z,0). 
  \een
Then  \eqref{valGlim} can be written as
\ben
{\cal G}_x(H_{k,\ell}) &=&kb_1(x)+\ell b_2(x)-{k^2 \over 2}c_{11}(x)-{\ell^2 \over 2}c_{22}(x)-k\ell c_{12}(x)+ \int_V \overline{H_{k,\ell}}^{h_0}(K(z,w))\nu_{\e}(dw),
\een
where, recalling that $\beta_{\e}= \sigma_{\e}^2 + \int_{(-1,\infty)} h_{\e}^2(w)\nu_{\e}(dw)$, 
\ben
c_{11}(x)&=&z(1-z)+\beta_{\e}\left( \frac{\partial p}{\partial w} (z,0)\right)^2+ \int_V   p_2(z,w) \nu_{\e}(dw) \\
&=&\sigma_{1,1}^2(x)+\sigma_{1,2}^2(x)+\int_V \big(h_0^1(K(z,w))\big)^2\mu(dw), \\
c_{22}(x)&=& \beta_{\e}  =\sigma_{2,2}(x)^2+\int_V \big(h_{0}^2(K(z,w))\big)^2\mu(dw),\\
c_{12}(x)&=& \beta_{\e}  \frac{\partial p}{\partial w} (z,0)+\int_{(-1,\infty)} h_{\e}(w) q(z,w) \nu_{\e}(dw)
= \sigma_{12}(x)\sigma_{2,2}(x)+\int_V h_0^1h_{0}^2(K(z,w))\mu(dw).
\een
Thus  ${\bf (H2)}$ holds for any $H=H_{k,\ell}\in {\cal H}$.

\me    
We can now apply Theorems \ref{thm-tightness}  and \ref{identification} for  tightness and identification and conclude.
\end{proof}
 \subsection{Pathwise  uniqueness and convergence in law}
To get the uniqueness for \eqref{EDSS}, we will use  the  pathwise uniqueness result from  Li-Pu \cite{li}. 
    \begin{cor}
  \label{WF-conv}  Let us assume that Assumption ${\bf A}$ holds and that  the function  $z\to p(z,w)$ is non-decreasing for  any $w\in (-1,+\infty)$.
  Then   the sequence of processes $\,\left(Z^{N}_{[N.]}/N,  S^N_{[N.]}\right)_N\,$ converges in law   in $
  \mathbb D([0,\infty),[0,1]\times \R)$
to the unique strong solution $(Z,Y)$ of \eqref{EDSS}.
 \end{cor}
 The monotonicity assumption on $p$ is natural regarding the model since the more individuals carry an  allele in a generation, the more this allele should be carried in the next generation.
\begin{proof} In order to apply Theorem \ref{main}, let us first show that  {\bf (H3)} holds.
The pathwise uniqueness of the process $Y$ is well known. Let us focus on the first equation of  \eqref{EDSS} and prove the pathwise uniqueness of the process $Z$. \\
First, we rewrite the SDE for $Z$ as
\ben Z_t &=&Z_0+\int_0^t \widetilde{b_{1}}(Z_s)ds + \int_0^t \sqrt{Z_s(1-Z_s)} dB_s^{\D} + \sigma_{\e}\int_0^t {\partial p \over \partial w} (Z_s,0) \,d B^{\e}_s  \\
&& \qquad  \int_{(-1,\infty)\setminus [-1/2,1]} (p(Z_{t-},w) - Z_{t-})  N(dt,dw)  +\int_{[-1/2,1]} (p(Z_{t-},w) - Z_{t-}) \widetilde N(dt,dw) 
\een
with
\ben
\widetilde{b}_1(z)&=&\left(\alpha_{\e}- \int_{(-1,\infty) \setminus [-1/2,1]}  h_{\e}(w)\nu_{\e}(dw)\right) {\partial p \over \partial w} (z,0)+{\sigma_{\e}\over 2}{\partial^2 p \over \partial w^2 } (z,0)\\
&&\qquad \qquad +\int_{[-1/2,1]} \left(p(z,w)-z-  h_{\e}(w)\frac{\partial p}{\partial w} (z,0)\right)\nu_{\e}(dw).
\een

We are  in the conditions of application of Theorem 3.2 in  \cite{li}.  Indeed, we observe first that  $\widetilde{b_1}$ is Lipschitz since $p\in C^3([0,1],(-1,\infty))$ and
$$\sup_{w\in [-1/2,1], z\in [0,1]} \bigg\vert\frac {\partial }{\partial z}  \ \bigg\{p(z,w)-z-  h_{\e}(w)\frac{\partial p}{\partial w} (z,0) \bigg\}\bigg\vert/w^2<\infty.$$
We   remark  also that  the Brownian part of \eqref{EDSS} writes  $$\sqrt{Z_t(1-Z_t)} dB_t^{\D} + \sigma_{\e}{\partial p \over \partial w} (Z_t,0) \,d B^{\e}_t= \sqrt{Z_{t}(1-Z_{t}) + \sigma_{\e}^2 \left( \frac{\partial p}{\partial w} (Z_{t},0)\right)^2} dW_{t} = \sigma(Z_{t}) dW_{t},$$ with $W$ Brownian motion since $B^{\D}$ and $B^E$ are two independent Brownian motions. We easily prove that for any $z_{1}, z_{2} \in [0,1]$, $|\sigma(z_{1})-\sigma(z_{2})|^2\leq L |z_{1}-z_{2}|$ for some constant $L>0$. \\
Finally
$\nu_{\e}((-1,\infty)\setminus[-1/2,1])<\infty$ and  $z\in [0,1]\rightarrow (p(z,w)-z)/w$ is uniformly Lipschitz for $w\in [-1,2,1]$ since its first derivative is bounded, so there exists $L>0$ such that 
$$  \int_{(-1,\infty)\setminus[-1/2,1]} ([p(z_1,w) - z_1]-[p(z_2,w) - z_2])^2  \nu_{\e}(dw)  \leq L\vert z_1-z_2\vert$$
for any $z_1,z_2 \in [0,1]$.
Then all the required assumptions for \cite{li} Theorem 3.2 are satisfied and we get the pathwise uniqueness of the solution of  \eqref{EDSS}. 
\end{proof}
\subsection{Example}
\label{ex.p}
\me  We consider the following  main example 
\be
\label{exemple-p}p(z,w) = \frac{z(1+w)}{z(1+w) + 1-z},\ee
where the environment $w$  acts as the selection factor. By construction, this selection coefficient $w$  is larger than $-1$. The particular
case when the environment is non-random, i.e. $E_k^N=s/N$ a.s. for some real number $s\in (-1,+\infty)$, yields the classical Wright-Fisher process with weak selection. It is well known that in this case, the 
processes $(Z^N_{[N.]})_{N}$ converge in law to   the Wright-Fisher diffusion with selection coefficient $s$ whose equation is given by $dZ_{t} = \sqrt{Z_t(1-Z_t)} dB_t + s Z_{t}(1-Z_{t})dt$. Here we 
generalize this result for random independent identically distributed environments.

\bi First, we observe that
 $$ {\partial p \over \partial w} (z,0)= z(1-z)\ ;\    {\partial^2 p \over \partial w^2} (z,0) = -  2 z^2(1-z).$$ 
and
\be
\label{equab}
b_1(z)= \alpha_{\e} z(1-z) -  \sigma_{\e} z^2(1-z)+\int_{(-1,\infty)} \Big(\frac{wz(1-z)}{zw+1}-  h_{\e}(w)z(1-z)\Big)\nu_{\e}(dw).
\ee
Under Assumption {\bf A}, we can apply Corollary \ref{WF-conv} to obtain the proposition stated below.

\begin{prop}
 The sequence of processes $\,\left(Z^{N}_{[N.]}/N,  S^N_{[N.]}\right)_N\,$ converges   in $\mathbb D([0,\infty),[0,1]\times \R)$
and the limit of the first coordinate  is  the unique strong solution  $Z$ of \be
\label{WFex}
Z_t &=&Z_0+\int_0^t b_{1}(Z_s) ds+ \int_0^t \sqrt{Z_s(1-Z_s)} dB_s^{\D} +\sigma_{\e} \int_0^tZ_{s}(1-Z_{s})  d B^{\e}_s\nonumber \\
&&\qquad \qquad \qquad \qquad +  \int_0^t\int_{(-1,+\infty)} {w Z_{s-} (1-Z_{s-})\over 1+ w Z_{s-} } \widetilde N(ds,dw).
\ee
\end{prop}
 In particular if $\sigma_{\e}=0$ and $\nu_{\e}=0$, we recover the classical Wright-Fisher diffusion with deterministic selection $\alpha_{\e}$. This extension allows us to consider small random fluctuations (asymptotically Brownian)   and punctual dramatic advantage of the selective effects.

\section{Continuous State Branching Process with interaction in  L\'evy environment}
\label{CSBPLEI}

In this section, we are interested in  large population approximations of  population dynamics with random environment and interaction. We    generalize in different directions the classical  convergence of Galton-Watson processes to \emph{Continous State Branching processes} (CSBP), see for example \cite{Grimvall, Lamp, CLU}.  We focus on 
models where the environment and the interaction  mainly affect  the mean of the reproduction law  and thus   modify the drift term of the CSBP by addition of stochastic and nonlinear terms.  Our method based on Section \ref{gene}  allows us to obtain new statements both for  convergence of discrete population models and for existence of solutions of SDE with jumps,  as detailed in the following theorems. In particular we  obtain a discrete population model   approximating  the so-called  \emph{CSBP  with interaction in L\'evy environment} (BPILE) for large populations. \\
The CSBPs  in random environment or with interaction have recently been subject of a large attention. More precisely, we refer to
\cite{PP} for existence of the solution of the associated SDE with general assumptions,  \cite{BS, BPS,  Li} for approximations and study of some classes of CSBP in random environment 
(without interaction), \cite{BaP, PD, PW} for CSBP with interaction (without random environment) and  \cite{Rosen} for diffusion approximations in the continuous context.   


\subsection{The discrete model}

\bi Let us now describe our framework. We scale the  population size  by the integer   $N\geq 1$ and we consider
 a  sequence $(v_{N})_{N}$ which tends to infinity with $N$.  As in Section 3, we introduce for any $N$ a sequence of independent identically distributed real-valued  random variables
 $(E^N_{k})_{k\geq 0}$ with same law as $E^N$. The asymptotical behavior of   $(E^N_{k})_{k\geq 0}$ is similar as the one in Section 3 (Assumption {\bf A})  but the scaling parameter is now $v_{N}$. 
 As in the previous section,  $h_{\e}$ denotes a truncation function   defined on $(-1,+\infty)$.  

 \begin{assumptionA1} Let us consider $\alpha_{\e} \in \R$, $\sigma_{\e} \in [0,\infty)$ and 
$\nu_{\e}$ 	a  measure on $(-1,\infty)$ such that
 \be
\label{hypnuE}
\int_{(-1,\infty)} (1 \wedge w^2) \ \nu_{\e}(dw)<\infty.
\ee
 Writing $\beta_{\e}= \sigma_{\e}^2 + \int_{(-1,\infty)}  h_{\e}^2(w) \nu_{\e}(dw)$, we assume that
$$
  \lim_{N\to \infty } v_N\, \E(h_{\e}(E^N))=\alpha_{\e};  \  \lim_{N\to \infty } v_N \,\E(h_{\e}^2(E^N))=\beta_{\e}; \ 
 \lim_{N\to \infty } v_N\, \E(f(E^N))=  \int_{(-1, \infty)} f(w) \, \nu_{\e}(dw),
$$
for any $f$  vanishing in a neighborhood of zero.
\end{assumptionA1}

We  also consider the associated random walk defined by
\be S^N_0 = 0, \  &&S^N_{k+1} = S^N_{k}+ E^N_{k}  \quad (k\geq 0). \nonumber \label{def-env}
\ee  
We recall  as in Section 3 that {\bf A1} is equivalent to the convergence
of the random walk  $S^N_{[v_N.]}$  to the  L\'evy process $Y$  with characteristics $(\alpha_{\e}, \beta_{\e}, \nu_{\e})$ defined in \eqref{environnement}.
We reduce the set of jumps to   $(-1,\infty)$ to avoid degenerated cases when a catastrophe below $-1$ could kill all the population in one generation. \\

\me  Let us fix $N$. We assume that given a population size $n$ and an environment $w$, each individual reproduces independently at generation $k$   with the same reproduction law $L^N(n,w)$.
We thus introduce   random variables $Z_{0}\geq 0$ and $\,L^N_{i,k}(n,w)$ such that the family of random variables $(Z_{0}, (L^N_{i,k}(n,w),  n\in \mathbb{N}, w\in (-1,+\infty)),  E^N_j ;  i, k\in \mathbb{N}^*, j\in \mathbb N)\,$ is independent
and for each $n\in \mathbb{N}, w\in (-1,+\infty)$, the random variables
$L^N_{i,k}(n,w)$ are all distributed as $L^N(n,w)$ for $i,k\geq 1$. We also assume that the function $\,L^N_{i,k}$ defined on $\Omega\times \mathbb{N}\times (-1,+\infty)\,$ endowed by the product $\sigma$-field is measurable. \\
The population  size  $Z^N_{k}$ at generation $k$ is  recursively defined   as follows, 
\be
 Z^N_0=[NZ_0], \ &&  Z^N_{k+1} = \sum_{i=1}^{Z_{k}^N} L_{i,k}^N(Z^N_{k},E^N_{k}) \quad \forall k\geq 0. \label{def-proc}
\ee

\me We investigate the convergence in law of the process $\Big(\big(\frac{1}{N}Z^{N}_{[v_Nt]},  S^N_{[v_Nt]} \big), t\in [0,\infty) \Big)$.  We cannot apply directly our general result to
 $Z^N_{[v_N.]}$. Indeed, the (associated) characteristics of the first component are not bounded. Moreover, scaling limits of $Z^N$ can lead to explosive processes, as already happens in the Galton-Watson case.  Therefore, we first study  the convergence  of the process
\be
\label{X}
X^N_k=\left(\exp(-Z^{N}_{k}/N),  S^N_{k}\right) \qquad (k\in \mathbb N)
\ee
 in $\mathbb{D}(\mathbb{R}_{+}, [0,1]\times \R)$ where  the state space of the first coordinate has been compactified.
Following  the notation of Section \ref{gene}, we introduce for $x = (\exp(-z),y)\in (0,1]\times \R$ the quantity
\be
F^N_x= \bigg(\exp\bigg(-\frac{1}{N}\sum_{i=1}^{[Nz]} \big(L_{i}^N ([Nz],E^N)-1\big) - z \bigg), y+E^N\bigg),
\ee
and observe that for any $z\in \mathbb N/N$, conditionally on $X_k^N=(\exp(-z),y)$, the random variable $X_{k+1}^N$ is distributed as $F^N_x$.

\me We   now apply  the theoretical framework   developed in Section \ref{gene}.
Let us introduce $\chi=(0,1]\times \R$ and
 $\mathcal U= [-1,1]\times (-1, \infty)$ and for $u=(v,w)\in \mathcal U$, we define
 \be
 h(u)=h(v,w)=(v,1-\exp(-w)), \qquad h_0(u)=h_0(v,w)=(v,h_{\e}(w))
 \ee
respectively as the specific function and the truncation function.  We choose the functional  space $\mathcal H$ defined by 
$$\mathcal H=\{ H_{k,\ell} : k\geq 1, \ell \geq 0 \} \cup \{ H_{\ell} :  \ell \geq 1\}, $$
where for any $u=(v,w)\in {\cal U}$,
$$H_{k,\ell}(u)= v^k\exp(-\ell w) \quad \text{and} \quad H_{\ell}(u)=1-\exp(-\ell w).$$
The fact that $\mathcal H$ satisfies  ${\bf (H1)}$ is a consequence of  the local Stone-Weierstrass Theorem on $[-1,1]\times [-1,\infty)\setminus\{(0,0)\}$ (cf. Appendix \ref{localSW}).  For any $k,\ell\geq 0$ and  $x= (\exp(-z),y)\in (0,1] \times  \mathbb{R}$, we have
\ben
{\cal G}^N_x (H_{k,\ell})&=& v_N \E\Big(H_{k,\ell}\Big(\exp\Big(-\frac{1}{N}\sum_{i=1}^{[Nz]} (L_{i}^N([Nz],E^N)-1) - z \Big)-\exp(-z), E^N\Big)\Big)\\
&=&e^{-kz}v_N\E\left(\left(e^{-\frac{1}{N}\sum_{i=1}^{[Nz]} (L_{i}^N([Nz],E^N)-1) }-1\right)^ke^{-\ell E^N}\right).\een
Let us set
\be
\label{noteP}
P_k^N(z,w)=\E\left(e^{- \frac{k}{N}(L^N([Nz],w)-1)}\right)^{[Nz]}-1
\ee
and
\be
\label{cond-A} A^N_{j,\ell}(z)=  v_N \E\left(P^N_j(z ,E^N) e^{-\ell E^N}\right).\ee
The presence of the term $\,-1\,$ in \eqref{noteP} may look strange at first glance,  but it ensures that $P_k^N\rightarrow 0$ as $N\rightarrow \infty$. 
Using the  binomial expansion and  by independence of the reproduction random variables conditionally on $E^N$, we obtain that
\be
\label{expGN}
{\cal G}^N_x (H_{k,\ell})
=
e^{-kz} \sum_{j=0}^{k} \binom{k}{j}(-1)^{k-j} A^N_{j,\ell}(z)
\ee
for $k\geq 1$, since $\sum_{j=0}^k \binom{k}{j}(-1)^{j-k}=0$ . We also obtain that  for $\ell\geq 1$,
$${\cal G}^N_x (H_{\ell})=v_N\E(1-\exp(-\ell E^N)).$$

\bi The  convergence of   $A^N_{j,\ell}$ characterizes the effect of the reproduction law, in a case of density dependence and random environment.
The uniform convergence and boundedness of $\exp(-kz)A^N_{j,\ell}(z)$ will ensure the tightness of $X^N_{{[v_{N}.]}}$ by Theorem \ref{thm-tightness}.
The continuity of the limiting functions is involved in the identification of the characteristic triplet of the limiting semimartingales.
Finally, their  representation  as solutions of a    stochastic differential equation  and the associated  uniqueness will  yield the convergence (Theorem \ref{main}).\\

\begin{rem}
In the case of Galton-Watson processes, $E^N=0, L^N(z,w)=L^N$ and $A^N_{j,\ell}(z)$ becomes
$$A^N_{j}(z)= v_N P^N_j(z ,0)=v_N \left(\E\left(e^{- \frac{j}{N}(L^N-1)}\right)^{[Nz]}-1\right).$$ 
We observe then that $A^N_{j}(z)\sim v_N[Nz]\E(1-\exp(-j(L^N-1)/N))$  as $N\rightarrow \infty$. 
It can easily be proved that the uniform convergence of $e^{-jz}A^N_{j}(z)$ is equivalent to the convergence of $v_NN\E(g((L^N-1)/N))$ for $g$ truncation function, its square or null in a neighborhood of zero. Thus this uniform convergence is equivalent to the
classical necessary and sufficient condition for convergence in law of Galton-Watson processes \cite{Grimvall, BS}.
\end{rem}

In the next section, we generalize this criterium to reproduction random variables depending on the population size  and the environment.

\subsection{Tightness}
\label{tgid}
We first prove the tightness of the processes $(X^N_{[v_{N}.]})_{N}$ defined in \eqref{X} by assuming the uniform convergence of the characteristics.

\begin{assumptionA1'} Let the characteristics $A^N_{j,\ell}$ be  defined in \eqref{cond-A}. For any  $1\leq j\leq k$ and $\ell\geq 0$,  there exists a  bounded function $A_{j,k,\ell}$ such that
$$  e^{-kz}A^N_{j,\ell}(z)    \stackrel{N\rightarrow \infty}{\longrightarrow}  A_{j,k,\ell}(z)$$
uniformly for $z\geq 0$. 
\end{assumptionA1'} 
Then we  state a tightness criterion for the original scaled process in the state space $[0,\infty]\times \R$ endowed with a distance $d$ which makes it compact and then  Polish, say $d(z_{1},z_{2})=\vert \exp(-z_{1})-\exp(-z_{2})\vert$ for $z_1,z_2\in [0,\infty]$ with the convention $\exp(-\infty)=0$.
\begin{thm}  \label{tensionditCSBP}
Under Assumptions {\bf A1} and {\bf A1'}, the sequence of processes 
$$\bigg(\big(\frac{1}{N}Z^{N}_{[v_Nt]},  S^N_{[v_Nt]} \big), t\in [0,\infty) \bigg) $$
is tight in  $\mathbb D(\R_{+}, [0,\infty]\times \R)$.
\end{thm}
\begin{proof}  For  $\ell\geq 1$, it  follows from Assumption {\bf A1} that
\be
v_{N}\, \E\left(1-e^{-\ell E^N}\right)  \stackrel{N\rightarrow \infty}{\longrightarrow}  \gamma_{\ell}^E=\alpha_\e\, z - \frac{1}{2}\,\sigma^2_{\e} \, z^2 +  \int_{(-1,+\infty)}  \big( 1-e^{-z w}- z h_{\e}(w) \big) \nu_{\e}(dw),
\label{conEnv}
\ee
since  $1 - e^{- \ell w} =  \ell h_{\e}(w)-\frac{1}{2} \ell^2h_{\e}^2(w)+ \kappa(w)$, 
where $\kappa(w)=o(w^2)$ is continuous bounded.   \\
Then we can define  ${\cal G}_.$ on $H_{\ell}$ for $\ell \geq 1$ as
\be
\label{GxHl}
{\cal G}_x (H_{\ell})=\gamma_{\ell}^{\e},
\ee
 for any $x\in  {\cal X}=(0,1]\times \R$. Let us now define ${\cal G}_.$ for   $H_{k,\ell}\in {\cal H}$  and $ k\geq 1, \ell \geq 0$. We set
\be
\label{GxHkl}
{\cal G}_x (H_{k,\ell})= \sum_{j=0}^{k} \binom{k}{j}(-1)^{k-j}A_{j,k,\ell}(z).
\ee
for  $x=(e^{-z},y)\in  {\cal X}$. Using Assumption {\bf A1'} and  \eqref{expGN},
we obtain that     $$\lim_{N\to \infty}\sup_{x\in {\cal X}} \left| {\cal G}_{x}^N(H) - {\cal G}_{x}(H) \right| = 0$$
 for any $H\in{\cal H}$. Moreover ${\cal G}_{.}(H)$ is bounded  by {\bf A1'}  and  Hypothesis {\bf (H1.3)} is satisfied.
The tightness of $(X^N_{[v_N.]})_N$  is then a consequence of  Theorem \ref{thm-tightness} and  yields the result.
\end{proof}

\subsection{Identification}

We now aim  at identifying the limiting values of  $(X^N_{[v_N.]})_N$ as diffusions with jumps. We are interested in models where the environment and the interaction affect the mean reproduction law. The limiting process will be called  \emph{CSBP with interaction in a L\'evy environment}  (BPILE).

\me We introduce  a  truncation function  $h_{\D}$ on the  state space $(0,+\infty)$,  parameters $\alpha_{\D} \in \R$ and   $\sigma_{\D}\geq 0$ and  a $\sigma$-finite measure $\nu_{\D}$ on $(0,+\infty)$ such that
\be
\label{hypnud}
\int_0^{\infty} ( 1\wedge z^2) \nu_{\D}(dz)<+\infty.
\ee
\me We also  consider a locally Lipschitz function $g$ defined  on $\R^+$ such that 
\be
e^{-z}z\, g(z)  \stackrel{z\rightarrow\infty}{\longrightarrow} 0.
\label{hypG}
\ee
The function $g$ models the interaction between individuals. In the applications to  population dynamics, the most relevant functions will be  polynomial.

\bi We provide now the scaling assumption on the reproduction random variable $L^N$ so that the limiting values  of  $Z^N/N$ can be identified to a BPILE. This assumption will become more explicit and natural through the identification and examples of the next sections.
\begin{assumptionA2} Setting  for $z\geq 0$,
\ben \label{defgam}
\gamma_{z}^{\D}&=&\alpha_{\D}\,z - \frac{1}{2}\,\sigma_{\D}^2 \,z^2 +  \int_{(0,+\infty)} \big(1-e^{-zr} - z h_{\D}(r)\big) \nu_{\D}(dr), \\
\label{express}
  \gamma_{z}^{\e}&=&\alpha_\e\, z - \frac{1}{2}\,\sigma_{\e}^2 \, z^2 +  \int_{(-1,+\infty)}  \big( 1-e^{-z w}- z h_{\e}(w) \big) \nu_{\e}(dw),
  \een
we assume that for any $1\leq j\leq k$ and $\ell\geq 0$,
\be
\label{A2exemple}
\sup_{z\geq 0} e^{-kz}\big\vert  A^N_{j,\ell}(z) \  +  \ j z\, g(z)+ \gamma_{j}^{\D}\,z   +  \gamma_{j z+\ell}^{\e}-\gamma_{\ell}^{\e}\big\vert \stackrel{N\rightarrow \infty}{\longrightarrow}  0, 
 \ee
 \noindent where $A^N_{j,\ell}$ has been defined in \eqref{cond-A}.
\end{assumptionA2}

\begin{rem}
\label{remarques}
(i) In Appendix \ref{exemplee}, we provide an explicit construction of  a family of  random variables
$L^N(z,e)$ satisfying {\bf A2},  in the case $\beta_{\D}=0$. \\
(ii)  We believe that the pointwise convergence induced by ${\bf A2}$ is actually necessary for the convergence of the process $Z^N/N$ to a BPILE. It does not seem sufficient in general since some integration argument is involved. Uniformity in ${\bf A2}$  provides a sufficient condition. It can be proved  for many classes of reproduction laws 
via uniform continuity, using  monotone or convex arguments  or boundedness of derivative on compact sets, see  examples.\\
(iii) Finally, let us remark that we  only need to prove the previous convergence for $z\in \mathbb N /N$ in {\bf A2},  using the definition of $A^N_{j,\ell}(z)$ and the uniform continuity of the limit. It will be  more convenient for examples.
\end{rem}

\me \noindent We observe that under  Assumption {\bf A2}, Assumption {\bf A1'} is satisfied with 
$$A_{j,k,\ell}(z)=e^{-kz}\left(-  j z\, g(z)- \gamma_{j}^{\D}\,z   +\gamma_{\ell}^{\e}-  \gamma_{j z+\ell}^{\e}\right).$$
Indeed, this expression is bounded      using \eqref{hypG} and the boundedness of  $\exp(-kz)\gamma_{j z+\ell}^{\e}$, since  $\vert \gamma_{j z+\ell}^{\e} \vert \leq C_{\ell,j}(z+z^2+e^{jz/2}z^2\beta_{\e}+e^{jz}\nu_{\e}(-1,-1/2))$ for $j\leq k$ and $k\geq 1$.  \\
Therefore the tightness holds by Theorem \ref{tensionditCSBP}. 

\noindent Moreover we can   simplify the expression \eqref{GxHkl} of the limiting characteristic ${\cal G}_{x}$, which writes
\be
\label{GxHkl2}
{\cal G}_x (H_{k,\ell})=e^{-kz} \sum_{j=0}^{k} \binom{k}{j}(-1)^{k-j}\left(-  j z\, g(z)- \gamma_{j}^{\D}\,z   +\gamma_{\ell}^{\e}-  \gamma_{j z+\ell}^{\e}\right)
\ee
for $x=(e^{-z}, y)$. For that purpose, we denote
$$f_{z}(u) = 1-e^{- z u}$$ 
and observe that
\be
 \sum_{j=0}^{k} \binom{k}{j}(-1)^{k-j} j &=& \delta_{1,k} \label{I1} \\
 \sum_{j=0}^{k} \binom{k}{j}(-1)^{k-j} \,j^2 &=& 2 \delta_{2,k} + \delta_{1,k} \label{I2}\\
 \sum_{j=0}^{k} \binom{k}{j}(-1)^{k-j} f_j(u) &=&  (-1)^{k+1}f_1(u)^k \label{I3} \\
 \sum_{j=0}^{k} \binom{k}{j}(-1)^{k-j} f_{jz+\ell}(u) &=& (-1)^{k+1} e^{-\ell u} f_z(u)^k. \label{I4}
 \ee
For $k\geq 3$, it follows from \eqref{I3} and \eqref{I4} and straightforward computation that   
\be
\label{Gk3}{\cal G}_x (H_{k,\ell})=(-1)^ke^{-kz} \left( \int_{(-1,+\infty)} e^{-\ell w} (f_z(w))^k \nu_{E}(dw)  +  z\,\int_{(0,+\infty)}  (f_1(r))^k \nu_{\D}(dr)\right).\ee
For $k = 2$, computation using \eqref{I2} leads to
\be
{\cal G}_x( H_{2,\ell})&=&e^{-2z} \Bigg\{  z^2 \beta_{E} + \int_{(-1,+\infty)} \big(e^{-\ell w} (f_z(w))^2 - z^2 h^2_{E}(w) \big) \nu_{E}(dw) + z \beta_{\D} \nonumber \\
&&\qquad \qquad \qquad \qquad + \, z \int_{(0,+\infty)} 
\big(f_1^2(r) - h^2_{\D} (r)\big) \nu_{\D}(dr)\Bigg\}. \label{H2l}
\ee
Similarly \eqref{I1} implies that
\be
\label{cas1l}
{\cal G}_x (H_{1,\ell})=e^{-z} \Big\{\gamma_{\ell}^E - \gamma_{z+\ell}^E 
     - z g(z)  -z \gamma_{1}^{\D}\Big\}.\ee
\me

To identify the limiting SDE,  we have to find the drift and variance terms and the jump measures in \eqref{eds}, from the expressions  \eqref{GxHl}, \eqref{Gk3}, \eqref{H2l} and \eqref{cas1l}. 

\bi We first remark that for  $k\geq 3, \ell \geq 0$, $ H_{k,\ell}=\overline{ H_{k,\ell}}$  with the notation  introduced in Lemma \ref{lem-decomp}. We work by identification  for $x=(e^{-z},y)\in (0,1]\times\R$ using   \eqref{Gk3}. We thus define the measure $\mu$ on $V=[0,+\infty)\times\R$ by
\be
\label{mu}\mu(d\theta, dr) = \un_{\theta \leq 1, r>-1}\, d\theta \, \nu_{\e}(dr) + \un_{\theta>1, r>0}\,  d\theta \, \nu_{\D}(dr),\ee
and the image function $K=(K_1,K_2)$  by
\be
\label{K} K_{1}(x,\theta, r) =  - e^{-z}. \bigg(  f_{z}(r)\, \un_{\theta \leq 1} +  f_{1}(r)\, \un_{1<\theta \leq 1+z}\bigg) \ ; \ K_{2}(x, \theta, r) =  r\, \un_{\theta \leq 1}.\ee
Then   $H_{k,\ell}$ satisfies ${\bf (H2.2)}$ for    $k\geq 3, \ell \geq 0$. 

\bi Moreover it is easy to find  $b_{2}$ and $ \sigma_{2,2}$ so   that  $H_{\ell}$   satisfies ${\bf (H2.2)}$ for $\ell \geq 1$ using that
$$H_{\ell}(u)=\ell  h_{\e}(w)-\frac{\ell^2}{2} h_{\e}(w)^2+  \overline{H_{\ell}}(u), \qquad   \overline{H_{\ell}}(v,w)=f_{\ell}(w)-\ell  h_{\e}(w)-\frac{\ell^2}{2} h_{\e}(w)^2$$
for  $u=(v,w)$ and \eqref{GxHl}.
Indeed,  by identification  and from \eqref{h2}, we set   for $x=(e^{-z},y)$
\be
\label{b} b_{2}(x) = \alpha_{\e}\  ; \ \sigma_{2,2}(x) = \sigma_{\e}.\ee

\bi Let us now consider the functions $H_{2,\ell}$ $(\ell \geq 0)$. Note that for $u=(v,w)$, we have
$$H_{2,\ell}(u) = v^2e^{-\ell w}= h_{\D}^2(v)+ \overline{ H_{2,\ell}}(u), \qquad \overline{ H_{2,\ell}}(u)= v^2(e^{-\ell w} -1) + v^2-h_{\D}^2(v).$$ 
The fact that ${\bf (H2.2)}$ is satisfied for $H_{2, \ell}$ comes from \eqref{H2l} for the left hand side and for the right hand it is given by a direct computation of 
$$\sigma_{1,1}(x)^2
 + \sigma_{1,2}(x)^2 + \int_{V} K_{1}^2(x,\theta, r) \mu(d\theta,dr) + \int_{V} \overline{H_{2,\ell}}(K(x,\theta,r)) \mu(d\theta,dr),$$
 where  $K$ is  defined from \eqref{K}.  Using $\overline{H_{2,\ell}}(K(x,\theta,r)) = K_{1}(x,\theta,r)^2(e^{-\ell K_{2}(x,\theta,r)} -1)$, the condition writes    for $x=(e^{-z},y)$,  $$\sigma_{1,1}(x)^2
 + \sigma_{1,2}(x)^2 = e^{-2z} (z \sigma_{\D}^2
 + z^2  \sigma_{\e}^2).$$
 
  \me It remains to check {\bf (H2.2)} for  $H_{1,\ell}$, with
$$ H_{1,\ell} (u) = ve^{-\ell w} = v(1- \ell h_{\e}(w)) + \overline{H_{1,\ell}}(u), $$
where $\overline{H_{1,\ell}}(u)=v(\ell h_{\e}(w)-f_{\ell}(w)) =o(|u|^2)$.  
Using \eqref{cas1l},  we have \ben
{\cal G}_x (H_{1,\ell})&=&e^{-z}\bigg( - \alpha_{\e}z - z\gamma_1^{\D} - zg(z)+\,  {z^2\over 2}\sigma_{\e}^2 +  \ell z \sigma_{\e}^2 \\
&&\quad +  \int_{(-1,+\infty)}f_{z} f_{\ell}(w) \nu_{\e}(dw) + \int_{(-1,+\infty)} (zh_{\e}(w)-f_{z}(w))\nu_{\e}(dw)\bigg).\een

As a conclusion,  both sides of \eqref{h2} coincide for $H\in \cal H$  by setting for any $x=(e^{-z},y)\in (0,1]\times \R$,
\be
\label{expb1}
b_1(x)= e^{-z}\bigg(  - \alpha_{\e}z - z\gamma_1^{\D} - zg(z)+\,  {z^2\over 2}\sigma_{\e}^2 + \int_{(-1,+\infty)} (zh_{\e}(w)-f_{z}(w))\nu_{\e}(dw)\bigg) \label{defb1}
\ee
and $b_{2}(x) = \alpha_{\e}$ and $K,\mu$ defined  by \eqref{K}  and \eqref{mu} and
\be
\label{S}
\sigma_{1,1}(x) = -  \sqrt{z}  \sigma_{\D} e^{-z}\ ;\ \sigma_{1,2}(x) = -  z  \sigma_{\e} e^{-z}\ ;\  \sigma_{2,1}(x) = 0\ ;\ \sigma_{2,2}(x) =  \sigma_{\e},\ee
and for any $x\in \{0\}\times \R$ and $(\theta,r)\in V$,
\be
\label{enzero}
&& b(x)=(0,\alpha_{\e}),   \quad \sigma_{11}(x)= \sigma_{21}(x)=\sigma_{12}(x)=0,  \quad \sigma_{22}(x)=\sigma_{\e}, \\
\label{enzeroo}
&& \qquad K_{1}(x,\theta,r)=0, \quad K_{2}(x, \theta, r) =  r\, \un_{\theta \leq 1}.
\ee

\bi The general  identification result for the exponential transformation of the processes can then be stated as follows, with $h_0(v,w)=(v,h_{\e}(w))$.
\begin{thm}
\label{idcsbp}
Under Assumptions {\bf A1} and {\bf A2}, the sequence of processes 
$$ \left( \left( \exp\left(-\frac{1}{N}Z^{N}_{[v_Nt]}\right),  S^N_{[v_Nt]} \right) : t\in [0,\infty)\right) $$
is tight in  $\mathbb D([0,\infty), [0,1]\times \R)$ and any limiting value $X \in \mathbb D([0,\infty), [0,1]\times \R)$ is a weak solution of the following two-dimensional stochastic differential equation 
 \be
 \label{eds2}
 X_{t}&=&X_{0} + \int_{0}^t b(X_{s}) ds + \int_{0}^t \sigma(X_{s}) dB_{s} + \int_{0}^t\int_{V} h_{0}(K(X_{s-},v)) \tilde N(ds, dv)\nonumber\\
 && \hskip 3cm + \int_{0}^t\int_{V} (Id - h_{0})(K(X_{s-},v)) N(ds, dv),
 \ee
 where $X_{0}=(\exp(-Z_0),0)$, $N$ is Poisson point measure with intensity $ds\mu(dv)$ on $\R^+\times V=[0,+\infty)^2\times\R$  and  $B$ is a two-dimensional Brownian motion and $Z_0,B,N$ are independent. The function $b=(b_{1}, b_{2})$, the matrix $\sigma$, the measure $\mu$ and the image function $K$ have been defined in \eqref{mu}-\eqref{enzeroo}.  \end{thm}

\begin{proof}[Proof of Theorem \ref{idcsbp}]
We already know that {\bf (H1)} is a consequence {\bf A2}. Let us check that   {\bf (H2)} is satisfied.
We first  prove the continuity {\bf (H2.1)}  of $x\to {\cal G}_{x}(H)$ for any $H\in \mathcal H$ and its extension to $\overline {\cal X}$. Recalling \eqref{GxHkl2}, we need to prove that $z\in [0,\infty) \rightarrow \gamma_{\ell}^{\e}-\gamma_{j z+\ell}^{\e}$
is continuous and  $\exp(-jz)(\gamma_{\ell}^{\e}-\gamma_{j z+\ell}^{\e})\rightarrow 0$ as $z\rightarrow \infty$. Indeed, the continuity can be obtained from the bound $\vert 1-e^{-(jz+\ell) w}- (jz+\ell)  h_E(w) \vert \leq C (1\wedge w^2)$ for any $z\in [z_0,z_1]\subset [0,\infty)$, while 
the limit as $z\rightarrow \infty$ can be proved using Lemma \ref{chiant} in Appendix and $\nu_{\e}(-1,-1+\varepsilon)\rightarrow 0$ as $\varepsilon\rightarrow 0$.
That allows us to prove that {\bf (H2.1)} is satisfied. \\
 Our choice of parameters   in \eqref{mu}- \eqref{enzeroo} ensures  that ${\bf (H2.2)}$ is satisfied for any $H_{\ell}$ and $H_{k,\ell}$. 
Applying  Theorem \ref{identification} to $X^N$ allows us to conclude.
\end{proof}

Let us now write  explicitly the stochastic differential equation \eqref{eds2} for  $X_{t}= (X^1_{t}, Y_{t})$ : 
 \ben
 \label{eds3}
 dX^1_{t}&=&X^1_{t} \log X^1_{t}\bigg( \alpha_{\e} + \frac{\sigma_{\e}^2}{2}\log X^1_{t}  + \, g(-\log X^1_{t}) + \alpha_{\D}- \frac{\sigma_{\D}^2}{2}\bigg)dt\\
&& - X^1_{t} \bigg(\int_{(-1,+\infty)} (1 - e^{w \log(X^1_{t})}+ \log X^1_{t}\,h_{\e}(w))\nu_{\e}(dw) \\
&& \qquad \qquad \qquad  \qquad - \log X^1_{t} \int_{(0,+\infty)} (1 - e^{-r}-h_{\D}(r))\nu_{\D}(dr)\bigg)dt\\
 &&+  \sigma_{\e} X^1_{t}  \log X^1_{t} dB^E_{t}  - \sigma_{\D} X^1_{t}  \sqrt{-\log X^1_{t}} dB^{\D}_{t}  -  \int_{(-1,+\infty)} X^1_{t-}(1-e^{w\log (X^1_{t})}) \widetilde N^{\e}(dt,  dw)\\
 &&\qquad  \qquad \qquad  - \int_{(0,+\infty)^2} \un_{\theta\leq -\log X^1_{t-}} X^1_{t-} (1-e^{-r})\widetilde N^{\D}(dt, d\theta, dr)\\
 dY_{t}&=& \alpha_{\e} dt + \sigma_{\e} dB_{t}^{\e} + 
\int_{(-1,+\infty)} h_{\e}(w) \widetilde N^{\e}(dt, dw)+ 
\int_{(-1,+\infty)} (w-h_{\e}(w) )N^{\e}(dt, dw),
 \een
where $B^{\e}$ and $B^{\D}$ are Brownian motions, $N^{\D}$ and $N^{\e}$ are Poisson Point measures respectively on $[0,\infty)\times (0,\infty)$ and on $[0,\infty)\times (-1,\infty)$ with intensity $dt\nu_{\D}(du)$ and $dt\nu_{\e}(dw)$ and $Z_0, B^{\e},B^{\D},N^{\D}$ and $N^{\e}$ are independent.

\bi
Using It\^o's formula (see  \cite{IK}), a straightforward  computation leads to  the equation satisfied by $Z_{t}= - \log X^1_{t}$.  More precisely, we define   the explosion time  $T_{exp}$ by
$$T_{exp}=\lim_{\varepsilon \rightarrow 0+} \inf\{t\geq 0 ; X_t^1 \leq \varepsilon\}=\lim_{a\rightarrow +\infty} \inf\{t\geq 0 ; Z_t \geq a\}\in [0,+\infty].$$ 
We obtain 
\be
\label{example}
Z_t &=& Z_0+\alpha_{\D}\int_0^{t} Z_sds+\int_0^{t}  Z_{s-}dY_s +  \int_0^tZ_s g(Z_s) ds +  \sigma_{\D} \int_0^{t}  \sqrt{Z_s} dB^{\D}_s + \\
&&  \int_0^{t} \int_{(0,+\infty)^2} {\un}_{\theta\leq Z_{s-}}h_{\D}(r) \widetilde N^{\D}(ds, d\theta, dr)+ \int_0^t\int_{(0,+\infty)^2}  {\un}_{\theta\leq Z_{s-}}(r - h_{\D}(r))  N^{\D}(dt, d\theta,dr).  \nonumber
\ee
on the time interval $[0,T_{exp})$ and $Z_t=+\infty$ for $t\geq T_{exp}$.\\
    When $T_{exp}=+\infty$ almost surely,  the process is said to be  \emph{conservative} (or non-explosive).  Grey's condition gives a criteria for CSBP, which has  been recently extended  to CSBP in random L\'evy environment in \cite{ZL}.   \\

We have thus proved the tightness of the process and identified the limiting values of $(X^N_{[v_N.]})_N$ as weak solutions of a SDE. Uniqueness of the SDE \eqref{example} (Hypothesis {\bf H3}) has to be proven to conclude for the convergence. From the pioneering works of Yamada and Watanabe, several results have been obtained for pathwise uniqueness relaxing the Lipschitz conditions on coefficients. In particular, general results for positive processes with jumps have been obtained in \cite{FL, li} and used in random environment, see in particular \cite{PP}. This technique allows us to conclude for strong uniqueness before explosion. Here, the process may explode in finite time, which is  already the case for classical CSBP  and in our framework, explosion can also be due to cooperation or random environment.  This leads us to consider two cases. In te first case,  we obtain a convergence in law on the state space $[0,\infty]$ under an additional regularity assumption on the drift term close to infinity. This 
result extends the classical criterion for convergence of Galton-Watson processes, adding  both  random environment and interaction.  In the second case,  we obtain the convergence of  $Z^N_{[v_N.]}$ in $[0,\infty)$ when the limiting values of the sequence of processes are non-explosive. We observe that it also extends  results
 of \cite{BS}  to L\'evy environment with infinite variation and of \cite{PD} by relaxing moment assumptions for interaction. \\
The pathwise uniqueness   of the SDE allows us  to capture limiting processes where infinity is either absorbing or non-accessible. Other situations are interesting, where  infinity is regular and uniqueness in law could be invoked. In particular, we refer to  \cite{Foucart} for a
criterion for reflection at infinity of CSBP with quadratic competition and  \cite{Andreasal} and \cite{BK} for similar issues.  

\subsection{Explosive CSBP with interaction and random environment}
\label{explo}

\me In this section, the process may be non-conservative, i.e. $T_{exp}$ may be finite.  In order to obtain the strong uniqueness and following  \cite{FL, Li},
we consider the following assumption concerning the regularity   of the drift term.
\begin{assumptionA3}
There exist continuous functions $r$, $b_r$ and $b_{d}$  such that for any $z \in [0,\infty)$,
\be
\label{hypb1}
e^{-z}\left(zg(z)- \frac{\sigma_{\e}^2}{2} z^2+ \int_{[-1/2,1]} (1-e^{-zw}-zh_{\e}(w))\nu_{\e}(dw)\right)=b_r(z)+b_{d}(z), 
\ee
with $r$ non-negative, non-decreasing and concave, $\int_0^. 1/r(z)dz =\infty$, $\vert b_r(-\log(u))-b_r(-\log(u'))Ê\vert \leq r(\vert u-u'\vert)$ for any $u,u'\in (0,1]$ and $b_{d}$  non-increasing.
\end{assumptionA3}
\begin{thm} 
\label{CSBPexplosv}
We assume that  {\bf A1},  {\bf A2} and {\bf A3} hold. \\
Then there exists a unique strong solution $(Z,Y)\in \mathbb D([0,\infty),[0,\infty]\times  \R)$ of \eqref{environnement} and \eqref{example} and 
$$ \left( \left(\frac{1}{N}Z^{N}_{[v_Nt]},  S^N_{[v_Nt]} \right) : t\in [0,\infty) \right) \Rightarrow ((Z_t,Y_t) : t\in [0,+\infty))$$
in  $\mathbb D([0,\infty), [0,\infty]\times \R)$, where $[0,\infty]$ is endowed with $d(z_1,z_2)=\vert \exp(-z_1)-\exp(-z_2)\vert$.
\end{thm}

\begin{proof}[Proof of Theorem \ref{CSBPexplosv}] We first remark that the convergence in law  of $(X^N_{[v_N.]})_N$
 in $\mathbb D([0,\infty),[0,1]\times \R)$
implies the weak convergence of $(Z^{N}_{[v_N.]}/N,  S^N_{[v_N.]})$ to $(-\log(X^1),Y)$ in $\mathbb D([0,\infty), [0,\infty]\times \R)$, where $[0,\infty]$ is endowed with $d$ and $-\log(0)=\infty$. \\
We recall from the previous section that $X^N$ satisfies ${\bf (H1)}$ and ${\bf (H2)}$. To apply Theorem \ref{main}, it remains to check that $X$ defined in \eqref{eds2} is unique in law. \\
Let us prove that under  {\bf A3}, pathwise uniqueness holds for $X$ in $\mathbb{D}([0,T], [0,1]\times \R)$.
First, the second component $Y$ of $X$ is a L\'evy process  and the pathwise uniqueness is well known. 
Second, the equation for the first component $X^1$ writes
\ben
X_t^1&=&X_0^1+\int_0^t \widetilde{b}_1(X_s^1)ds+\int_0^t\sigma(X_s^1)dW_s+\int_0^t\int_{V\setminus V_0} K^1(X_{s-}^1,v) N(ds,dv)\\
&& \qquad \qquad + \int_0^t\int_{V_0} K^1(X_{s-}^1,v) \widetilde{N}(ds,dv)
\een
where $\sigma(u)=\sqrt{\sigma_{1,1}(u)^2+\sigma_{1,2}(u)^2}$ and  $W$ is a Brownian motion independent of $X_0^1$ and of the Poisson point measure $N$. The set  $V_0$ is defined as $V_{0}=[0,1]\times [-1/2,1] \cup (1,\infty]\times (0,\infty)$. For  $x_1=\exp(-z)$ and $\rho= \alpha_{E} + \gamma_1^{\D}- \int_{(-1,\infty)-[-1/2,1]} h_{\e}\nu_{\e}$,
$$\widetilde{b}_1(x_1)= e^{-z}\bigg(  -z\rho  - zg(z)+\,  {z^2\over 2}\sigma_{\e}^2 + \int_{[-1/2,1]} (zh_{\e}(w)-f_{z}(w))\nu_{\e}(dw)\bigg). $$

\me
We first observe that
$\mu(V\setminus V_0)<\infty$.  Moreover  , 
combining \eqref{expb1} and \eqref{hypb1}, we have
$$\widetilde{b}_1(x_1)=x_1\log(x_1) \rho -b_r(-\log(x_1))-b_{d}(-\log(x_1))=\widetilde{b}_r(x_1)+\widetilde{b}_{d}(x_1),$$
where $\widetilde{b}_{d}=-b_{d}(-\log .)$ is non-decreasing and $\widetilde{b}_r$ satisfies 
$\vert \widetilde{b}_r(x_1)-\widetilde{b}_r(\widetilde{x}_1)\vert\leq  \widetilde{r}(\vert x_1-\widetilde{x}_1\vert)$ for $x_1,\widetilde{x}_1\in [0,1]$,  with $\int_0^. 1/\widetilde{r}(z) dz=\infty$
and $\widetilde{r}$ non-decreasing and concave. Indeed using Lemma \ref{lemcoriace} in Appendix, one can take $\widetilde{r}(y)=r(y)+ Cy  +C_1r_1(y)$, with $r_1(x_1)=-x_1\log(x_1)$ and $C,C_1$ well chosen.\\
Then  we easily check that $\sigma^2$ is Lipschitz continuous and $\vert \sigma(y)- \sigma(y')\vert^2 \leq \vert \sigma(y)^2- \sigma(y')^2\vert$
and 
 $y\rightarrow y+K^1(y,v)$ is non-decreasing.  \\
 Finally,
\ben
\int_{V_0} (K^1(y,v)-K^1(y',v))^2\mu(dv) &=&  g_1(y,y')\int_{\R^{+}}(e^{-r}-1)^2\nu_{\D}(dr)\\
&&\qquad +\int_{[-1/2,1]} (g_2(y,w)-g_2(y',w))^2\nu_{\e}(dw),
\een
where for any $y, y' \in (0,1]$,
\be
\label{expg}
g_1(y,y')=\min(-\log(y),-\log(y')) (y-y')^2+ \min(y,y')^2\vert \log(y) -\log(y') \vert
\ee
(with a null extension at $0$) 
and 
\be
\label{autre}
g_2(y,w)=u(e^{\log(y)w}-1).
\ee
Using now Lemma \ref{bornefonction} in Appendix and the integrability assumptions on $\nu_{\D}$ and $\nu_E$, there exists $L>0$ such that 
$$\int_{V_0} (K^1(y,v)-K^1(y',v))^2\mu(dv)\leq L\vert y-y'\vert.$$
Then we can apply  Theorem 3.2 in \cite{li} and conclude by observing that $X_t=0$ for $t\geq T_{exp}$ by pathwise uniqueness.
\end{proof}
 Recently, 
 Pardoux and  Dram\'e \cite{PD} have proven the convergence of some continuous time and discrete space processes to CSBP with interaction. Here we relax their conservative assumption and extend to random environments and to general classes of reproduction laws, in a discrete time setting.

\paragraph{Application to Galton-Watson processes with cooperative effects.}
\label{secGWcop}

\me  Note that  Theorem \ref{CSBPexplosv} allows us to recover the convergence in law of the Galton-Watson processes $(\widehat Z^N_{[v_{N}.]})_{N}$  defined as in \eqref{def-proc} with the reproduction laws  $L^N\in \mathbb N$ satisfying: \be
&&  \lim_{N\to \infty }  v_NN\,\E(h_{\D}((L^N-1)/N)
))= \alpha_{\D}; \qquad
  \lim_{N\to \infty }  v_NN \, \E(h^2_{\D}((L^N-1)/N
))=\beta_{\D}; \nonumber \\ 
&&
\qquad \qquad \qquad  \lim_{N\to \infty }    v_NN \, \E(f((L^N-1)/N
))  =  \int_{0}^{\infty} f(v) \nu_{\D}(dv),\label{condtripGW}
\ee
 for any continuous bounded  function $f$  vanishing in a neighborhood of $0$, 
where $h_{\D}$ is  a truncation function, $\alpha_{\D}\in \R$, $\int_{(0,\infty)} (1\wedge v^2) \, \nu_{\D}(dv)<\infty$,
$\beta_{\D}=\sigma_{\D}^2+ \int_{(0,\infty)} h_{\D}^2 \, \nu_{\D}$ and $\sigma_{\D}\geq 0$.

\me  The limiting process is the (possibly explosive) CSBP with characteristics $(\alpha_{\D},\beta_{\D},\nu_{\D})$  solution of the stochastic differential equation 
\be
\label{edsgw}
\widehat Z_t &=& Z_0+\alpha_{\D}\int_0^{t} \widehat Z_sds +  \sigma_{\D} \int_0^{t}  \sqrt{\widehat Z_s} dB^{\D}_s + \\
&& \ +\int_0^{t} \int_{(0,\infty)^2} {\bf 1}_{\theta\leq \widehat Z_{s-}}h_{\D}(r) \widetilde N^{\D}(ds, d\theta, dr)+ \int_0^t\int_{(0,\infty)^2}  {\bf 1}_{\theta\leq \widehat Z_{s-}}(r - h_{\D}(r))  N^{\D}(dt, d\theta , dr),  \nonumber
\ee
where  $N^{\D}$ is a Poisson measure with intensity $dtd\theta\nu_{\D}(dr)$.

 \me As a new application of Theorem \ref{CSBPexplosv}, we extend the convergence above  by  taking into account a cooperative effect. In this case, the interactions  prevent the use of the classical generating function tool.    The reproduction random variable  $L^N(n)$ depends on the total population size $n$ and we set  
 \be
 \label{repro-cons}L^N(n)=L^N+ \mathcal E^N(n),\ee
where for each $n\geq 0$, $\mathcal E^N(n)\in \{0,1\}$   is a  Bernoulli random variable independent of $L^N$ and 
\be
\label{defBern}
\p\left(\mathcal E^N(n)=1\right)=\frac{g(n/N)\wedge v_N}{v_N}
\ee
for some function  $g \in \mathcal C^1([0,\infty), [0,\infty))$.  
 The process $Z^N$ is defined as in \eqref{def-proc} with this reproduction random variable $L^N(n)$.
 
 \me We obtain the following convergence result.
\begin{prop} We assume that $v_N\rightarrow \infty$ and that  \eqref{condtripGW}, \eqref{repro-cons} and \eqref{defBern} hold. We also assume that $z\rightarrow \exp(-z)zg(z)$ is non-increasing for $z$ large   enough and goes to $0$ as $z\rightarrow \infty$.\\
Then $(Z^{N}_{[v_N.]}/N : t\geq 0)$ converges in  $\mathbb D([0,\infty), [0,\infty]\times \R)$ to the unique strong solution   $Z$  of
\be
\label{exampleGWcoop}
Z_t &=& Z_0+\alpha_{\D}\int_0^{t} Z_sds+\int_0^{t}  Z_s g(Z_s) ds +  \sigma_{\D} \int_0^{t}  \sqrt{Z_s} dB^{\D}_s  \\
&& \ +\int_0^{t} \int_{(0,\infty)^2} {\bf 1}_{\theta\leq Z_{s-}}h_{\D}(z) \widetilde N^{\D}(ds, dz, d\theta)+ \int_0^t\int_{(0,\infty)^2}  {\bf 1}_{\theta\leq Z_{s-}}(z - h_{\D}(z))  N^{\D}(dt, dz, d\theta) \nonumber
\ee
for $t<T_{exp}$ and $Z_t=+\infty$ for $t\geq T_{exp}$.
\end{prop}
The monotonicity assumption on $z\rightarrow \exp(-z)zg(z)$ is  chosen for sake of simplicity to obtain the   pathwise uniqueness. It captures in particular simple cooperative functions as $g(z)=cz^{\alpha}$  ($c> 0, \alpha>0$) or $g(z)=c+b(1-1/(1+z))$ ($c\geq 0, b>0$).\\
 We observe  that the limiting process $Z$ may be explosive, 
 due to the heavy tails of the reproduction random variable $L^N$ (i.e. the CSBP part is explosive) or 
 due to cooperative effects (note for instance that   $y'_t=y_tg(y_t)$ is explosive if  $g(z)=z^{\alpha}$, $\alpha>0$). \\
 Finally, we add that extensions of the last convergence to random environments are possible in several ways, in particular catastrophes can be added and {\bf A3} still holds. But if $\sigma_{\e}>0$, the function $g$ has to compensate the quadratic term so that {\bf A3} can be fulfilled. Otherwise,  other arguments have to be invoked and one may expect  to get uniqueness in law using quenched Laplace exponent (without interaction) or duality arguments.  \\
\begin{proof}
Let us introduce \be
\label{expcoop}
C_{j}^N(z)=v_N\left(\E\left(e^{-j\Sigma^N}\right)^{Nz}\left( \frac{g(z)\wedge v_N}{v_N}e^{-j/N} +\left(1-\frac{g(z)\wedge v_N}{v_N}\right)\right)^{Nz}-1\right).
\ee
By a Taylor expansion (developed in  Appendix \ref{GWcoop}), one can prove that \be
\label{pointdel}
\sup_{z; zN\in \mathbb{N}} e^{-kz}\big\vert  C_{j}^N(z) +  \ j z\, g(z)+ \gamma_{j}^{\D}\,z \big\vert\stackrel{N\rightarrow \infty}{\longrightarrow} 0.
\ee
 Assumption {\bf A2} is fulfilled for $z\in {\mathbb{N}\over N}$, which is enough as commented in Remark \ref{remarques}, while {\bf A1} is trivial  (no random environment). As $g \in \mathcal C^1([0,\infty), [0,\infty))$  and $\exp(-z)zg(z)$ is non-increasing for $z$ large   enough and goes to $0$ as $z$ goes to infinity,  there exist $b_r$ and $b_{d}$ such that
$$e^{-z}zg(z)=b_r(z)+b_{d}(z),$$
with $b_{d}$ non-increasing and $b_r(-\log(u))$ Lipschitz continuous such that  Assumption {\bf A3} is fulfilled. Indeed
there exists $z_0$ such that  $z\rightarrow e^{-z}zg(z)$ is non increasing for $z\geq z_0$ and one can take
$b_{d}(z)=e^{-z}zg(z)$ for  $z\geq z_0$ and $b_{d}$ constant for $z\leq z_0$, and  $b_r(z)= e^{-z}zg(z)-b_{d}(z)$.\\
We conclude using Theorem \ref{CSBPexplosv}.
 \end{proof}

\subsection{Conservative CSBP with interaction and random environment}

\bi We focus on the conservative case. Now $+\infty$ is not accessible and
the  pathwise uniqueness is  obtained without Assumption  {\bf A3}. 
\begin{thm} \label{cascon} We assume that  {\bf A1} and  {\bf A2} hold and  that any  solution of \eqref{example} is conservative, i.e. $T_{exp}=+\infty$ a.s.
Then  there exists a unique strong solution $(Z,Y)\in \mathbb D([0,\infty),  [0,\infty)\times \R)$ of \eqref{environnement} and \eqref{example} and 
$$ \left( \left(\frac{1}{N}Z^{N}_{[v_Nt]},  S^N_{[v_Nt]} \right) : t\in [0,\infty)\right) \Rightarrow ((Z_t,Y_t) : t\in [0,\infty))$$
in  $\mathbb D([0,\infty), [0,\infty)\times \R)$.
\end{thm}

\me 
  Theorem \ref{cascon} allows us to obtain various scaling limits to diffusions with jumps  due either to the environment or to demographic stochasticity. The conditions for tightness and identification are  very general. The conservativeness 
can be obtained  by different methods as  moment estimates or comparison   with a conservative CSBP or conservative CSBP in random environment when 
the process is  competitive or  with bounded cooperation.

\begin{proof}
Using that $T_{exp}=+\infty$ a.s., one can check that pathwise uniqueness holds for $\eqref{example}$. It can be achieved by  using the pathwise uniqueness for $Z$ obtained in  \cite{PP} before $T_{exp}$
or  by adapting the proof of  Theorem \ref{CSBPexplosv}. 
We recall from Theorem \ref{idcsbp} that weak existence also holds for \eqref{eds2} under  {\bf A1} and  {\bf A2}, so that both strong existence and weak uniqueness hold.\\
Then {\bf (H3)} is fulfilled and we can apply Theorem \ref{main} to $X^N$ and get the weak convergence of $\big(\exp(-Z^{N}_{[v_N.]}/N),  S^N_{[v_N.]}\big)$ to $X$ in $\mathbb D([0,\infty),[0,1]\times \R)$. Since  $T_{exp}=+\infty$, the weak convergence of $(Z^{N}_{[v_N.]}/N,  S^N_{[v_N.]})$ in $\mathbb D([0,\infty), [0,\infty)\times \R)$ and the pathwise uniqueness of $(Z,Y)$ follow, which ends up the proof.
\end{proof}

\me {\bf Application to logistic Feller diffusion  in a Brownian environment.} The next  example illustrates the result. We consider a reproduction law which takes into account logistic competition and small fluctuations of the environment. 
\begin{cor}
Assume that  $(E^N)_{N}$ are centered random variables such that $(\sqrt{N}E^N)_{N}$ is uniformly bounded  
 and has variance $\sigma_{\e}^2$. We define
$L^N\in \{0,1,2\}$   for $N$ large enough, $n\in  \mathbb N$ and $e\in (-1,\infty)$ by
\be
\label{logB}
\mathbb P(L^N(n,e)=0)={1\over 2}(\sigma_{\D}^2-e+g_N(n/N)), \quad  \mathbb P(L^N(n,e)=2)={1\over 2}(\sigma_{\D}^2 +e-g_N(n/N)),
\ee
where    $\sigma_{\D}\in (0,\sqrt{2})$,
$g_N(z)=\alpha_{\D}/N +c(z/N)\wedge (1/\sqrt{N})$ for $z\geq 0$ and $c\geq 0$ and $\alpha_{\D} \in \R$.\\
Then  $(Z_t^N/N : t\in [0,\infty))$ converges in law in $\mathbb D([0,\infty), \R\times [0,\infty))$ to the unique strong solution $Z$  of
$$Z_t=Z_0+\alpha_{\D}\int_0^t Z_sds -c\int_0^tZ_s^2ds +\sigma_{\e}\int_0^t Z_sdB_s^{\e}+\sigma_{\D}\int_0^t\sqrt{Z_s}dB_s^{\D},$$
where $B^{\e}$ and $B^{\D}$ are two independent Brownian motions.
\end{cor}
\begin{proof}
 Assumption ${\bf A1}$ holds with $v_N=N$, $\alpha_{\e}=0$, $\nu_{\e}=0$ and $\beta_{\e}=\sigma_{\e}^2$. Let us now prove that
 {\bf A2} holds. \\
 First, from  \eqref{logB}, we get
\ben
\E\left(e^{- \frac{j}{N}(L^N(n,e)-1)}\right)&=& 1-\frac{j}{N}(e-g_N(n/N))+\frac{j^2}{2N^2}\sigma_{\D}^2+o(1/N^2),
\een
where $o(1/N^2)$ is uniform with respect to $z$ and $e$.
Then, for any $z\in \mathbb N/N$,
\ben
P_j^N(z,e)&=&\E\left(e^{- \frac{j}{N}(L^N(Nz,e)-1)}\right)^{Nz}-1
\ = \ e^{Nz\left(-\frac{j}{N}(e+g_N(z))+\frac{j^2}{2N^2}\sigma_{\D}^2+o(1/N^2)\right)}-1\\
&=& -jz(e-g_N(z))+\frac{j^2}{2}z^2e^2+ \frac{j^2z}{2N}\sigma_{\D}^2+o(e^{jz}/N)
\een
by considering the cases $z\leq \sqrt{N}$ and  $z\geq \sqrt{N}$. We obtain that
for any $1\leq j\leq k$ and $\ell\geq 0$,
\ben
e^{-kz}N\E\big(P^N_{j}(z,E^N) \, e^{-\ell E^N}\big)&=&e^{-kz} 
\bigg( \left(jzNg_N(z)+ \frac{j^2z}{2}\sigma_{\D}^2 \right)\mathbb E\left(e^{-\ell E^N}\right)\\
&&  \qquad -jzN \mathbb E\left(E^Ne^{-\ell E^N}\right) + \frac{j^2}{2}z^2N \mathbb E\left((E^N)^2e^{-\ell E^N}\right)  \bigg)+o(1).
\een
Finally, $\sqrt{N}E^N$ is centered, bounded with variance $1$, so $\mathbb E\left(e^{-\ell E^N}\right)\rightarrow 1$ and
$N\E(f(E^N))\rightarrow \sigma_{\e}^2f''(0)/2$ for $f\in C^{b,2}_0$ when $N$ tends to infinity. In particular,  
$$N \mathbb E(E^Ne^{-\ell E^N}) \rightarrow -\ell \sigma_{\e}^2, \qquad N \mathbb E((E^N)^2e^{-\ell E^N}) \rightarrow \sigma_{\e}^2.$$
Writing $g(z)=cz$ and using that  $\gamma_j^{\D}=j\alpha_{\D}-\frac{j^2}{2}\sigma_{\D}^2$ and $\gamma_v^{\e}= \sigma_{\e}^2v^2/2$, we get 
$$
\sup_{z\in \mathbb N/N} e^{-kz}\bigg\vert  \mathcal C_{j,\ell}^N(z)+ \gamma_{j z+\ell}^{\e}-\gamma_{\ell}^{\e}-j z\, g(z)+ \gamma_j^{\D} z\bigg\vert \stackrel{N\rightarrow \infty}{\longrightarrow}  0. 
$$
since $\gamma_{j z+\ell}^{\e}-\gamma_{\ell}^{\e}= \sigma_{\e}^2(zj\ell +z^2j^2/2)$.
We recall from Remark 4.2(iii) that this uniform convergence then holds for $z\geq 0$
and {\bf A2}  is satisfied.\\
Finally,  a coupling with the Feller diffusion in Brownian environment ($c=0$, studied in \cite{BH}) allows us to prove that  the process $Z$ is conservative. The result is then an application of Theorem \ref{cascon}. \end{proof}

\section{Perspectives and multidimensional population models}
\label{Perspect}
The general results of Section \ref{gene} have been applied in the two previous sections to  Wright-Fisher processes in a L\'evy environment and Galton-Watson processes with interaction in a L\'evy environment  with jumps larger than $-1$. These generalizations of historical  population models
 were our original motivation for this work. The results of Section \ref{gene}  can actually be applied in  other interesting contexts. We mention here some hints in these directions and works in progress. \\

 First, we could consider environments which are non independent and identically distributed or not restricted to $(-1,\infty)$.\\
 This restriction to $(-1,\infty)$  allowed   to consider a functional space generated by the functions $\exp(-k.)$ $(k\geq 0)$ which are bounded on $(-1,\infty)$. To extend the results to  random walks converging to L\'evy processes with a jump measure
 $\nu$ on $\R$ such that $\int_{\R} (1\wedge w^2)\nu_{\e}(dw)<\infty$, one could consider the functional space of compactly supported functions
  $$\mathcal H=\{ (x,w)\rightarrow e^{-kx}f(w) : k\geq 1, f\in C^{\infty}_c(\R)\}\cup \{ (x,w)\rightarrow f(w) : f\in C^{\infty}_c(\R), f(0)=0\}$$
  for studying Wright Fisher in a L\'evy environment and
 $$\mathcal H=\{  (u,w)\rightarrow u^kf(w) : k\geq 1, f\in C^{\infty}_c(\R)\}\cup \{  (u,w)\rightarrow f(w) :  f\in C^{\infty}_c(\R), f(0)=0\}$$
 for studying branching processes with interaction in random environment. Indeed these spaces satisfy {\bf (H1.1,2)}. This would require to  check that {\bf (H1.3)} holds.\\
Such functional spaces could also help to study cases when the environment $E^N_k$ depends on $S^N_k$ and $S^N$ converges to a diffusion with jumps.
  \\

Second,  as explained in the introduction, we are more generally interested  in $k$-type  population models, where the population at generation $n$ is described by a vector $$Z^N_n=(Z^{1,N}_n, Z^{2,N}_n,\ldots,   Z^{k,N}_n),$$
where  $Z_n^{i,N}$ counts the number of individuals of type $i$ in generation $n$. The following processes  have attracted lots of attention in population dynamics framework :
$$Z^{i,N}_{n+1}=\sum_{\alpha=1}^k\sum_{j=1}^{F_N^{(\alpha)}(Z^N_n)} L^{N,\alpha}_{i,j,n}(Z^N_n).$$
Such processes allow to model competition, prey-predators interactions, sexual reproduction, mutations ....
Some examples  have been well studied,  as multitype branching processes, controlled branching processes or bisexual Galton-Watson processes, see e.g. respectively  \cite{Mode},   \cite{GMD} and \cite{Alssurvey}.\\
One way to obtain the scaling limits is to consider the compactified proces
$$X^N=\left(\exp(-Z^{1,N}_n), \exp(-Z^{2,N}_n),\ldots,  \exp(- Z^{k,N}_n)\right)$$
and to use the functional space
$$\mathcal H=\big\{ (u_1, \ldots, u_k) \rightarrow u^{i_1}\times u^{i_k} : (i_1, \ldots, i_k) \in \mathbb N^k\setminus (0,\ldots 0)\big\}.$$
Indeed $\mathcal H$ satisfies Assumption {\bf (H1.1,2)} and the exponential transformation combined with this functional space may allow to exploit the independence structure of the model as for extended branching processes in Section \ref{CSBPLEI}. Some work will then be  required to check that Assumption {\bf (H1.3)} holds. Moreover uniqueness can be delicate.  In a work in progress, we consider bisexual Galton-Watson processes and their scaling limits to bisexual CSBPs under general conditions. It is also worth noticing that in the scaling limits, the nonlinearity or the environment can  impact the diffusion or jump terms, and not only the drift as for BPILE considered in Section \ref{CSBPLEI}. One could also prove limits to CSBP with L\'evy environment, where the jump measure associated with the demographical stochasticity (large jumps coming from the offsprings of one single individual, at a rate proportional to the number of individuals)
is impacted by the environment, see \cite{BS}, \cite{Li} for an example. 
 
Note also that we observe that one may want to go beyond the boundedness assumptions on the characteristics ${\cal G}^N$. This seems to be a challenging question but the approach developed here may be extendable. Indeed,
we obtain the boundedness assumptions in Section \ref{CSBPLEI} by a compactification of the state space using the function $z\rightarrow \exp(-z)$, which allows to consider explosive processes. 

The last point to mention is that our criteria concern  semimartingales in general. The Markov setting allows us to simplify the form of the characteristics ${\cal G}^N$ and to reduce the problem to analytical approximations,  nevertheless we could try to work with non  Markovian processes with similar techniques.

\section{Appendix}

\subsection{General construction of a discrete   random variable satisfying   {\bf A2}}
\label{exemplee}
We first consider the case $\sigma_{\D}=0$ and assume 
$E^N \in (-1+1/\sqrt{N}, \infty)$ for simplicity. We also introduce $g_N$ which converges to $g$ 
and such that \be
\label{cvv}
e^{-z}z g(z)\stackrel{z\rightarrow \infty}{\longrightarrow } 0, \qquad \sup_{z\geq 0} e^{-z}z\vert g_N(z)-g(z)\vert\stackrel{N\rightarrow \infty}{\longrightarrow } 0, \qquad   \sup_{z\geq 0} \frac{\vert g_N(z)\vert}{N^{1/3}}<\infty.
\ee
One can take for instance $g_N(.)=g(.)\wedge N^{1/3}$.
Let us define
$$m_N(n,e)=1+g_N(n/N)/N+\alpha_{\D}/N+e$$
and
 observe that $m_N(n,E^N)$ is a.s. positive for $N$ large enough. We  consider the reproduction random variable 
$A^N(n,e)\in \{[m_N(n,e)],[m_N(n,e)]+1\}$ defined by $\mathbb E(A^N(n,e))=m_N(n,e)$, i.e.
$$\mathbb P(A^N(n,e)=[m_N(n,e)])=p_N(n,e), \quad \mathbb P(A^N(n,e)=[m_N(n,e)]+1)=1-p_N(n,e), $$
with $p_N(n,e)=[m_N(n,e)]+1 -m_N(n,e)$.
For the large reproductions events, we also introduce $\Sigma^N \in \mathbb N$ independent of $(A^N(n,e) : n\geq 0, e\in (-1,\infty) )$ such that
$$
  \lim_{N\to \infty }  N^2 \,\E(h_{\D}(\Sigma^N
))= 0; \quad 
  \lim_{N\to \infty }  N^2 \, \E(h^2_{\D}(\Sigma^N
))=0; \quad
  \lim_{N\to \infty }    N^2 \, \E(f(\Sigma^N
))  =  \int_{0}^{\infty} f(v) \nu_{\D}(dv)$$
for $f$ continuous bounded and 
vanishing in a neighborhood of $0$.
The reproduction random variable $L^N$ is then defined by
$$L^N(n,e)=A^N(n,e)+N\Sigma^N $$
for $n\in \mathbb N$ and $e>-1$ and writing $\kappa_j^N=-\log(1-\E(f_j(\Sigma^N))$, we have for $z\in \mathbb N/N$,
\ben
&&\mathbb E\left(e^{- \frac{j}{N}(L^N(Nz,e)-1)}\right)\\
&&\qquad =\mathbb E\left(e^{- {j}{\Sigma^N}}\right) \mathbb E\left(e^{- \frac{j}{N}(A^N(Nz,e)-1)}\right)\\
&&\qquad =e^{-\kappa_j^N-je/N} \left( p(Nz,e)e^{-j([m_N(n,e)]-1-e)/N}+(1-p(Nz,e))e^{-j([m_N(n,e)]-e)/N}\right)\\
&&\qquad =e^{-\kappa_j^N-je/N} \left(1-\frac{j}{N}\left(m_N(Nz,e)-e-1\right)+\frac{\phi_N(z,e)}{N^2}\right)
\een
where $\phi_N$ is  bounded. By Taylor expansion, we obtain
\be
\big\vert \phi_N(z,e) \big\vert &\leq & c.\bigg(\left([m_N(Nz,e)]-e-g_N(z)/N-\alpha/N)(1-2([m_N(Nz,e)]-e)\right)\nonumber \\
&&\qquad \qquad \qquad  \qquad\qquad +([m_N(Nz,e)]-e)^2\bigg). \label{inegphiN}
\ee
Moreover $m_N(Nz,e)-e-1=\mathcal O(N^{-2/3})$ uniformly for $z,e$ and we obtain
\ben
P^N_j(z,e)&=&e^{-zN\kappa_j^N-jez} \left(1-\frac{j}{N}\left(m_N(Nz,e)-e-1\right)+\frac{\phi_N(z,e)}{N^2}\right)^{Nz}-1\\
&=&e^{-zN\kappa_j^N -jez -jz\left(m_N(Nz,e)-e-1\right)+z\psi_N(z,e)}-1\\
&=&e^{-zN\kappa_j^N-zj(g_N(z)+\alpha_{\D})/N}.e^{-jze-z\psi_N(z,e)}-1,
\een
where $N\psi_N(z,e)$ is continuous bounded and $N\vert \psi_N(z,e)\vert\leq c(1/N^{4/3}+\vert  \phi_N(z,e) \vert)$ and $c$ is a constant which may change from line to line.
Thus
$$
\mathbb E\big(P^N_{j}(z,E^N) \, e^{-\ell E^N}\big)=A_1^N(z)+A_2^N(z) + A_1^N(z)A_2^N(z)+A_3^N,$$
where
$$A_1^N(z)=e^{-zN\kappa_j^N-zj(g_N(z)+\alpha_{\D})/N}-1, \quad A_2^N(z)=\E\left(e^{-(jz+\ell)E^N-z\psi_N(z,E^N)}\right)-1$$
and  $A_3^N=\E(f_{\ell}(E^N))$. Assumption ${\bf A1}$ ensures that 
$v_NA_3^N$ converges to $\gamma_{\ell}^{\e}$ when $N$ tends to infinity (see \eqref{conEnv}  for details).  To conclude and prove \eqref{A2exemple}, we prove and combine the asymptotic results stated below.

\begin{lem} For any $j\geq 1$,
\ben
(i) &&\sup_{z\geq 0} e^{-jz}\big\vert NA_1^N(z)+z(\gamma_j^{\D}+\alpha_{\D} +g(z))\big\vert \stackrel{N\rightarrow \infty}{\longrightarrow}0\\
(ii) && \sup_{z\geq 0} e^{-jz}\big\vert NA_2^N(z)+\gamma_{jz+\ell}^{\e} \big\vert \stackrel{N\rightarrow \infty}{\longrightarrow} 0\\
(iii) && \sup_{z\geq 0} e^{-jz}\big\vert NA_1^N(z)A_2^N(z)\big\vert \stackrel{N\rightarrow \infty}{\longrightarrow} 0
\een
\end{lem}
\begin{proof}
(i)  First,  by Taylor expansion and using   that $g_N(z)/N^{1/3}$ is bounded, there exists $c>0$ such that  for any  $z\leq N^{2/3}$,
\ben
&&e^{-jz}\big\vert Ne^{-zN\kappa_j^N-zj(g_N(z)+\alpha_{\D})/N} +z(\gamma_j^{\D}+\alpha_{\D} +g(z)) \big\vert   \\
&&\qquad \qquad \leq c.e^{-jz}z\left(\vert N^2\kappa_j^{N}-\gamma_j^{\D}\vert+\vert g_N(z)-g(z)\vert+N^{-1/3}\right).
\een
The right hand side  goes to $0$ uniformly as $N\rightarrow \infty$. Second,
$$ e^{-jz}z(\gamma_j^{\D}+\alpha_{\D} +g(z)) \stackrel{z\rightarrow \infty}{\longrightarrow} 0, \quad \sup_{z\geq N^{2/3}, N\geq 1} e^{-jz} \big\vert N  e^{-zN\kappa_j^N-zj(g_N(z)+\alpha_{\D})/N} -1 \big\vert \stackrel{A\rightarrow \infty}{\longrightarrow} 0,$$
since for  $z\geq N^{2/3}$,
$$e^{-jz}\big\vert Ne^{-zN\kappa_j^N-zj(g_N(z)+\alpha_{\D})/N} -1 \big\vert \leq e^{-jz}N(e^{zc/N^{2/3}}+1)\leq Ne^{-N^{1/3}.(1-c/N^{2/3})}+Ne^{-N^{1/3}},$$
which goes to $0$. This proves $(i)$. 

Let us turn to $(ii)$. We first prove the uniform convergence on compact sets using convexity and simple convergence.
Indeed, recalling that  $\vert N\psi_N(z,e) \vert \leq c(1/N^{4/3}+\vert  \phi_N(z,e) \vert)$,
\ben
\big\vert A_2^N(z)+\E\left(f_{jz+\ell}(E^N)\right) \big\vert \leq \frac{c}{N}\left(1/N^{4/3}+ \E( \vert  \phi_N(z,E^N) \vert ) \right)
\een
for $z\in [0,A]$. By  {\bf A1}, $E^N$ goes in probability to $0$  and $\phi_N$ is bounded and $\phi_N(z,e)\rightarrow 0 $  as $e\rightarrow 0$ uniformly with respect to $z\in [0,A]$ from \eqref{inegphiN}. It turns out  that $\sup_{z\in [0,A]} \E( \vert  \phi_N(z,E^N) \vert )\rightarrow 0$ and
$$\sup_{z\in [0,A]} \big\vert NA_2^N(z)+N\E\left(f_{jz+\ell}(E^N)\right) \big\vert  \stackrel{N\rightarrow \infty}{\longrightarrow} 0.$$
Moreover, for any $z\geq 0$, $N\,\E\left(f_{jz+\ell}(E^N)\right)\rightarrow \gamma_{jz+\ell}^{\e}$ by Assumption {\bf A1} (see  again \eqref{conEnv}  for details) and the convergence is uniform on $[0,A]$ by convexity  of
$z\rightarrow N\E\left(f_{jz+\ell}(E^N)\right)$ and by continuity of $z\rightarrow \gamma_{jz+\ell}^{\e}$ (third Dini's theorem). 
It proves $(ii)$ on compacts sets. Let us now prove that  $\sup_{z\geq A,N\geq 1} \exp(-jz)\big\vert NA_2^N(z)\vert \rightarrow 0$ as $A\rightarrow \infty$.\\
Let us fix $\varepsilon>0$ and
\be
\label{deccompp}
A_2^N(z)=B_{\varepsilon}^N(z)+C_{\varepsilon}^N(z)
\ee
where 
$$B_{\varepsilon}^N(z)= \E\left(1_{E^N\geq -1+\varepsilon} e^{-(jz+\ell)E^N-z\psi_N(z,E^N)}\right)-1.$$
Recalling that $E^N\geq -1+1/\sqrt{N}$ and $N\psi_N$ bounded, we have 
$$
C_{\varepsilon}^N(z)= \E\left(1_{E^N<-1+\varepsilon} e^{-(jz+\ell)E^N-z\psi_N(z,E^N)}\right)
\leq   \mathbb P(E^N<-1+\varepsilon) e^{-(jz+\ell)(1-1/\sqrt{N}+c/N)}.$$
Thus,  the last part of Assumption {\bf A1} ensures that
$$ \lim_{\varepsilon\rightarrow 0} \sup_{N\geq c^2, z\geq 0} e^{-jz}NC_{\varepsilon}^N(z)=\lim_{\varepsilon\rightarrow0} \sup_{N\geq 1}  N\mathbb P(E^N<-1+\varepsilon) = \lim_{\varepsilon\rightarrow 0} 
\nu_{\e}(-1,-1+\varepsilon)=0.$$
Writing 
$$g_x(y)=f_{x}(y)-xh_E(y)= 1-e^{-xy}-xh_E(y), \quad R_N(z,e)=e^{-(jz+\ell)e}\left(e^{-z\psi_N(z,e)}-1\right),$$
we have 
\ben
B_{\varepsilon}^N(z)&=&\mathbb P(E^N<-1+\varepsilon)-(jz+\ell)\E\left(h_{\e}(E^N)1_{E^N \geq -1+\varepsilon}\right)-\E(g_{jz+\ell}(E^N))1_{E^N \geq -1+\varepsilon})\\
&&+\E\left(R_N(z,E^N)1_{E^N \geq -1+\varepsilon}\right).
\een
First, we recall that $\sup_N\mathbb P(E^N<-1+\varepsilon)\rightarrow 0$
as $\varepsilon\rightarrow 0$ and
 $N\E\left(h_{\e}(E^N)1_{E^N \geq -1+\varepsilon}\right)$ is bounded (actually convergent by Assumption {\bf A1}). Second, we prove that
$$\sup_{z\geq A, N\geq 1}N  e^{-jz}\E\left(\vert g_{jz+\ell}(E^N)\vert 1_{E^N\geq -1+\varepsilon}\right)\stackrel{A\rightarrow\infty}{\longrightarrow} 0$$
using that (see forthcoming Lemma \ref{chiant} for details)
$$\sup_{y\geq  -1+\varepsilon, \ y \ne 0} \frac{e^{-jz}}{(1-e^{-y})^2} \vert g_x(y) \vert  \stackrel{x\rightarrow\infty}{\longrightarrow} 0$$
and that $N\E((1-\exp(-E^N))^2)$ is bounded from {\bf A1}. Finally 
$$e^{-jz}N\vert R_N(z,e)1_{e\geq -1+\varepsilon} \vert \leq N\exp(-(\varepsilon -1/\sqrt{N})jz).\big\vert \exp(cz/N)-1\big\vert$$
ensures that $$\sup_{z\geq A,N} N\E\left(R_N(z,E^N)1_{E^N \geq -1+\varepsilon}\right)\stackrel{A\rightarrow\infty}{\longrightarrow} 0,$$ 
using again $\vert \exp(cz/N)-1\vert\leq c' z/N$ for $z\leq N$, while the  right hand is bounded by $z\exp(-\varepsilon jz/2)$ for $z\geq N$ and $N\geq 16/\varepsilon^2$ . Combing these estimates in  \eqref{deccompp} yields \\ 
$\sup_{z\geq A,N\geq 1} \exp(-jz)\big\vert NA_2^N(z)\vert \rightarrow 0$ as $A\rightarrow \infty$ and ends the proof of $(ii)$ by recalling that $\exp(-jz)\gamma^{\e}_{jz+\ell}\rightarrow 0$ as $z\rightarrow \infty$. \\

We finally prove $(iii)$. First,  
$$\sup_{z\geq 0} e^{-jz/2\sqrt N }\vert A_1^N(z) \vert \stackrel{A\rightarrow\infty}{\longrightarrow} 0$$ using
that $\exp(-jz/2\sqrt N )\vert A_1^N(z)\vert \leq c\exp(-jz/2\sqrt N)z/N^{2/3} $ for $z\leq N$ (and then one may use for $N$ large that 
 $z/N^{2/3}$ is small for  $z\leq N^{7/12}$  and that  $\exp(-z/2\sqrt N)$ is small for $N^{7/12}\leq z\leq N$) and 
 $\exp(-jz/2\sqrt N )\vert A_1^N(z)\vert\leq  \exp(-jz(1/2\sqrt N-c/N^{2/3}))$ for $z\geq N$. Second
 $$\sup_{N\geq 1,z\in [0,\infty)} e^{-jz(1-1/2\sqrt N)}N \E(A_2^N)<\infty,$$
 since
 $$ \E(A_2^N)\leq \mathbb P(E^N\leq -1/2)e^{(jz+\ell)(1-1/\sqrt{N})+zc/N}+ \E(A_2^N1_{E^N>-1/2}) $$
and $N\mathbb P(E^N\leq -1/2)$ is bounded from {\bf A1} and $N\E(A_2^N1_{E^N>-1/2})$ is bounded, following the point $(ii)$ and using the following slight modification of Lemma \ref{chiant}
$$\sup_{y> -1/2, \ y \ne 0, \ N\geq N_0} \frac{e^{-x(1-2/\sqrt{N})}}{(1-e^{-y})^2} \vert g_x(y) \vert  \stackrel{x\rightarrow\infty}{\longrightarrow} 0,$$
where $N_0$ is chosen such that $1-2/\sqrt{N_0}>1/2$.
\end{proof}

\subsection{Taylor expansion for a Galton-Watson process with cooperation}
\label{GWcoop}
Recalling \eqref{expcoop} and \eqref{defgam},
\ben
C_{j}^N(z)&=&v_N\left(\left(1-\gamma_j^{N,D}/Nv_N\right)^{Nz}\left( 1 -\frac{j}{N}\frac{g(z)\wedge v_N}{v_N} +\mathcal O\left(\frac{g(z)\wedge v_N}{N^2v_N}\right)\right)^{Nz}-1\right)\\
&=&v_N\left(e^{-z\gamma_j^{N,D}/v_N+\mathcal O(z/Nv_N^2)-jz(1\wedge (g(z)/v_N))(1+\mathcal O(1/N))}-1\right)
\een
for any $z\in \mathbb N/N$.
For $z$ such that $z+g(z)z\leq v_N$, we have $z/v_N\leq 1$ and $1\wedge (g(z)/v_N)=g(z)/v_N$ for $z\geq 1$. We make a Taylor expansion and get
$$e^{-jz}\big\vert C_{j}^N(z)+ \, z\gamma_j^{N,D}+ jzg(z)\big\vert\leq e^{-jz}c \left(\frac{1}{N}+\frac{zg(z)}{N} \right)
  $$
for some constant $c>0$. We obtain 
$$\sup_{z+g(z)z\leq v_N} e^{-jz}\big\vert  C_{j}^N(z)+ \, z\gamma_j^{\D}+ jzg(z)\big\vert \stackrel{N\rightarrow \infty}{\longrightarrow} 0.$$
To conclude, we observe that  $\min\{ z : z+g(z)z\geq v_N\}\rightarrow \infty$ as $N\rightarrow \infty$. Then   $\sup_{z+g(z)z\geq v_N} e^{-jz}\big\vert z\gamma_j^{N,D}+ jzg(z) \big\vert\rightarrow \infty$. Let us now prove that
$$\sup_{z+g(z)z\geq v_N} e^{-jz}v_N\big\vert  C_{j}^N(z)\big\vert \stackrel{N\rightarrow \infty}{\longrightarrow} 0.$$
Indeed, $g\geq 0$ and either $z\geq v_N/2$ and
$$  e^{-jz}v_N\big\vert  C_{j}^N(z) \vert \leq e^{-jz}v_Ne^{zc/v_N}\leq 2ze^{-zj/2}$$
for $N$ such that $j-c/v_N\geq j/2$ or $z\leq v_N/2$  and $v_N\leq 2zg(z)$ and there exists $c>0$ such that
$$  e^{-jz}v_N\big\vert  C_{j}^N(z) \vert \leq   e^{-jz} 2zg(z)\, e^{c}.$$
Recalling that $\min\{ z; z+g(z)z\geq v_N\}\rightarrow \infty$ as $N\rightarrow \infty$ and that  $zg(z)\exp(-z)\rightarrow 0$ as $z\rightarrow \infty$, we obtain  the desired result.

\subsection{Some technical results}
\label{technique}
\begin{lem} \label{chiant} For $x > 0$, let us consider
$$g_x(y)=1-e^{-xy}-xh_E(y).$$ 
We have 
$$\sup_{y> -1+\varepsilon, \ y \ne 0} \frac{e^{-x}}{(1-e^{-y})^2} \vert g_x(y) \vert  \stackrel{x\rightarrow\infty}{\longrightarrow} 0.$$
\end{lem}
\begin{proof}
Let $\mathcal V_0$ be an open finite interval containing $0$ such that
$h_{\e}(y)=y$ for $y\in \mathcal V_0$. There exists $C>0$ such that for any $y \not\in \mathcal V_0$, 
$$\frac{\vert g_x(y)\vert }{(1-e^{-y})^2} \leq C\,(1+x+e^{x(1-\varepsilon)})$$
since $h_{\e}$ and $1/(1-\exp(-y))$ are  bounded. The result follows on  the complementary set of $\mathcal V_0$.
Let us now consider $y \in \mathcal V_0$.  Assuming $\vert xy \vert \leq 1$, we get $\vert g_x(y)\vert \leq C\, x^2y^2$
and we conclude using that $y/(1-\exp(-y))$ is bounded on $(-1,\infty)$.\\
If  $y \in \mathcal V_0$ and  $\vert xy \vert \geq 1$, we have 
$$\frac{\vert g_x(y)\vert }{(1-e^{-y})^2}\leq C\left(\frac{\vert 1-e^{-xy}\vert }{y^2}+ \frac{x}{y}\right)\leq Cx^{2}\left(1+e^{x(1-\varepsilon)}\right),$$
which ends the proof. 
\end{proof}

\noindent Let us now prove  the forthcoming inequality \eqref{coriace}. 
\begin{lem}
\label{lemcoriace}
Let $r_1(x)=-x\log(x)$. Then for any $x, x' \in [0,1]$, 
\be
\label{coriace} \vert x\log(x)- x'\log(x')\vert\leq K \big(\vert x - x'\vert +  r_1(\vert x -x '\vert)\big)\ee
for some constant $K>0$.
\end{lem}

\begin{proof}
Let us first assume that $\min(x,x')\geq |x - x'|$. In this case, it is immediate that
$$ |x \log x - x' \log x'| \leq |x - x'| (1 + \log(|x - x'|))$$
by the mean value theorem.
We now assume that $0\leq x \leq |x - x'| \leq  x'$, which implies that $x'\leq 2 |x - x'|$.  We have 
\ben |x \log x - x' \log x'| &\leq& |x\log(x/x') +  (x - x') \log(x')|\\
&\leq & 
  |x\log(x/x')| + |\log(|x - x'|)|\, |x - x'| \\
& \leq&  x' - x + |\log(|x - x'|)|\, |x - x'|,
\een
using that $x/x'\in [0,1]$ and that the function $\alpha \in  [0,1]\rightarrow \alpha \log \alpha$ is bounded by some constant $C$.
We obtain that $ |x \log x - x' \log x'| \leq 2C |x-x'| + |\log(|x - x'|)|\, |x - x'|$, which ends the proof.
 \end{proof}

\begin{lem} \label{bornefonction} With notation \eqref{expg}
and \eqref{autre}, for any $x_1,\widetilde{x}_1 \in [0,1]$,
$$g_1(x_1,\widetilde{x}_1) \leq L\vert x_1-\widetilde{x_1} \vert$$
and for any $u \in [-1/2,1]$,
$$(g_2(x_1,u)-g_2(\widetilde{x}_1,u))^2 \leq  C\vert x_1-\widetilde{x_1} \vert   u^2.$$
\end{lem}
\begin{proof}
For the first inequality, one can use that $-x\log x$ is bounded for the first term in \eqref{expg}
and the mean value theorem for the second one.\\
Concerning the second inequality, we use
$$(g_2(x_1,u)-g_2(\widetilde{x}_1,u))^2\leq \big\vert g_2(x_1,u)^2-g_2(\widetilde{x}_1,u)^2\big\vert \leq \sup \vert (g_2^2(.,u))' \vert \vert Z_s\vert $$
and 
$$(g_2^2(.,u))'(x_1) =2x_1(e^{\log(x_1)u}-1)^2+ux_1e^{\log(x_1)u}2(e^{\log(x_1)u}-1).$$
The results then come from the inequality $\vert e^{\log(x_1)u}-1 \vert \leq \vert \log(x_1)\vert u$.
\end{proof}

\subsection{Stone-Weierstrass Theorem on locally compact space}
\label{localSW}
We recall here  
the local version of Stone-Weierstrass Theorem and  assume that the space $X$ is a locally compact Hausdorff space.  \\
Let $C_{0}(X,\mathbb{R})$ the space of real-valued continuous  functions  on $X$ which vanish at infinity, i.e.  given $\varepsilon>0$,  there is a compact subset $K$ such that $\|f(x)\|<\varepsilon$ whenever the point $x$ lies outside $K$. In other words, the set $\{x, \|f(x)\|\ge \varepsilon\}$ is compact. \\
 Let us consider a subalgebra $A$ of $C_{0}(X,\mathbb{R})$. 
Then $A$ is dense in $C_{0}(X,\mathbb{R})$ for  the topology of uniform convergence if and only if it separates points and vanishes nowhere.

\section*{Acknowledgments}

This work was partially funded by the Chair "Mod\'elisation Math\'ematique et Biodiversit\'e" of VEOLIA-Ecole Polytechni\-que-MnHn-FX
and by the  ANR ABIM  16-CE40-0001.


\begin{thebibliography}{20}

 \bibitem{Alssurvey} G. Alsmeyer.  \newblock Bisexual Galton Watson processes : a survey.  \newblock {\em Available via https://www.uni-muenster.de/Stochastik/alsmeyer/bisex(survey).pdf}.


\bibitem{BaP} M. Ba and E. Pardoux.
\newblock  Branching processes with interaction and a generalized Ray-Knight theorem. 
\newblock \emph{Ann. Inst. Henri Poincar\'e Probab. Stat}. 51 (2015), no.4, 1290--1313. 
  
\bibitem{BPS} V. Bansaye, J. C. Pardo Millan and C. Smadi. On the extinction of continuous state branching processes
with catastrophes.  \emph{Electron. J. Probab.} 18 (2013),  no.106.

\bibitem{BS}
V. Bansaye and F. Simatos.
\newblock On the scaling limits of {G}alton-{W}atson processes in varying
  environment.
\newblock \emph{Electron. J. Probab.} 20 (2014), no 75.

\bibitem{BK}
J. Bertoin and I. Kortchemsky.
\newblock Self-similar scaling limits  of Markov chains on the positive integers.
\newblock {\em  Ann. Appl. Probab.}  26 (2016), no.4, 2556--2595.


\bibitem{BH}
C. Boeinghoff and M. Hutzenthaler.
\newblock Branching diffusions in random environment.
\newblock {\em Markov Process. Related Fields}, 18 (2012), no.2, 269--310.

\bibitem{Bor}
K.~Borovkov.
\newblock A note on diffusion-type approximation to branching processes in
  random environments.
\newblock  {\em  Teor. Veroyatnost. i Primenen} 47 (2002), no.1, 183--188; translation in {\em Theory Probab. Appl.} 47 (2003), no.1, 132--138.

\bibitem{CLU}
M. E. Caballero, A. Lambert  and G. Uribe Bravo.
\newblock Proof(s) of the Lamperti representation of continuous-state branching processes 
{\em Probab. Surveys} 6 (2009), 62--89.

\bibitem{CPU}
M. E. Caballero, J.L. P\'erez  and G. Uribe Bravo.
\newblock A Lamperti-type representation of continuous-state branching processes with immigration
{\em  Ann. Probab.}  41 (2013), no. 3A, 1585--1627.


\bibitem{DL}  D. A. Dawson and Z. Li. \newblock Stochastic equations, flows and measure-valued processes.
\newblock {\em  Ann. Probab.} 40 (2012), no. 2, 813--857.

\bibitem{PD} I. Dram\'e and E. Pardoux. 
\newblock Approximation of a generalized CSBP with interaction.
\newblock {\em  Electron. Commun. Probab.}, (2018). 

\bibitem{EK}  S. N. Ethier and T. G. Kurtz. \newblock {\em Markov processes. Characterization and convergence.} 
\newblock Wiley Series in Probability and Mathematical Statistics: Probability and Mathematical Statistics, New York, 1986. 

\bibitem{Foucart} C. Foucart. Continuous-state branching processes with competition: Duality and Reflection at Infinity. \newblock {\em Available via }
https://arxiv.org/abs/1711.06827, 2017.

\bibitem{FL}  Z. Fu and Z. Li. \newblock Stochastic equations of non-negative processes with jumps.
\newblock
 {\em Stochastic Process. Appl.}, 120 (2010), no.3, 306--330.

 \bibitem{GMD}  M. Gonz\'alez, M.  Molina M. and I. Del Puerto.
 \newblock
 On L2-convergence of controlled branching processes with random control function.
\newblock {\em Bernoulli}, 11 (2005), no.1, 37--46.

\bibitem{Grimvall}
A. Grimvall.
\newblock On the convergence of sequences of branching processes.
\newblock {\em Ann. Probab.}  2 (1974), 1027--1045.

\bibitem{ZL} H. He, Z. Li, W. Xu. \newblock Continuous-state branching processes in Levy random environments 
\newblock {\em J. Theor. Probab.} (2018), p. 1--23. 

\bibitem{IK}
N. Ikeda and S. Watanabe.
\newblock \emph{Stochastic differential equations and diffusion processes. 2nd ed}.
\newblock North-Holland, 1989.

\bibitem{Jacod}
J. Jacod 	and A.S. Shiryaev.
\newblock \emph{Limit theorems for stochastic processes}.
\newblock 2nd Edition. Springer 2002.

\bibitem{Kosenkova}
T.I. Kosenkova.
\newblock Weak convergence of a series scheme of Markov chains to the solution of a L\'evy driven SDE. 
\newblock {\em Theory Stoch. Process.} 18 (2012), no.1, 86--100.

\bibitem{Kurtz78}
T.~G. Kurtz.
\newblock Diffusion approximations for branching processes.
\newblock In {\em Branching processes ({C}onf., {S}aint {H}ippolyte, {Q}ue.,
  1976)}, vol.~5 of {\em Adv. Probab. Related Topics}, p. 269--292.
  Dekker, New York, 1978.
 
 \bibitem{Andreasal}
A. Kyprianou, S. W. Pagett, T.  Rogers and J. Schweinsberg. A phase transition in excursions from infinity of the "fast'' fragmentation-coalescence process. \emph{Ann. Probab.} 45 (2017), no. 6A, 3829--3849.

 
 \bibitem{lam05}
A. Lambert.
\newblock The branching process with logistic growth. 
\newblock {\em Ann. Appl. Probab.}, 15  (2005), no.2, 150--1535.
 

\bibitem{Lamp}
J. Lamperti.
\newblock Continuous state branching processes.
\newblock {\em Bull. Amer. Math. Soc.}, 73 (1967), 382--386.

\bibitem{Lamp2}
J. Lamperti.
\newblock The limit of a sequence of branching processes.
\newblock {\em Z. Wahrscheinlichkeitstheorie und Verw. Gebiete}, 7 (1967), 271--288.
  
  \bibitem{Lampstable}
  J. W. Lamperti. \newblock Semi-stable stochastic processes.
  \newblock {\em  Trans. Amer. Math. Soc.} 104 (1962), 62--78.
  
  \bibitem{PW} V. Le, E. Pardoux and A. Wakolbinger. \newblock
Trees under attack: a Ray-Knight representation of Feller's branching diffusion with logistic growth.
\newblock{\emph{Probab. Theory Relat. Fields}},155 (2013), 583--619.

  \bibitem{li}
  Z. Li and F. Pu.
  \newblock Strong solutions of jump-type stochastic equations.
  \newblock \emph{Electronic Communications in Probability} 17 (2012), no.3, 1--13.

\bibitem{Li} P.-S. Li. A continuous-state nonlinear branching process.
\newblock { \em Available Arxiv  : https://arxiv.org/abs/1609.09593}

\bibitem{Mode} 
C. J. Mode. \newblock {\em Multitype branching processes. Theory and applications. }
\newblock Modern
Analytic and Computational Methods in Science and Mathematics, No. 34. American Elsevier Publishing Co., Inc., New York, 1971.

\bibitem{PP} S.  Palau and J-C. Pardo Millan. \newblock Branching processes in a L\'evy random environment. \newblock { \em  Acta Mathematica Applicandae,} 153 (2018), no.1,  55--79.

\bibitem{Pardoux}
E. Pardoux. \emph{Probabilistic models of population evolution. Scaling limits, genealogies and interactions. }  Mathematical Biosciences Institute Lecture Series. Stochastics in Biological Systems, 1.6. Springer, MBI Mathematical Biosciences Institute, Ohio State University, Columbus, OH, 2016.

\bibitem{Rosen} G. Rosenkranz. \newblock 
Diffusion approximation of controlled branching processes with random environments. 
\newblock {\em Stoch. Anal. Appl.} 3 (1985), 363--377. 
\end{thebibliography}
\end{document}